\documentclass[leqno,12pt]{article}
\usepackage{amssymb,amsmath,amsthm,enumerate }
\usepackage[latin1]{inputenc}
\usepackage[T1]{fontenc}
\usepackage{mathrsfs}
\usepackage{hyperref}

\usepackage[mmddyyyy]{datetime}

\newcommand{\Z}{\mathbb{Z}}
\newcommand{\R}{\mathbb{R}}
\newcommand{\C}{\mathbb{C}}

\newcommand{\Cs}{\ensuremath{C^*}}
\newcommand{\Csr}{\ensuremath{C_r^*}}
\newcommand{\E}{\ensuremath{\mathcal{E}}}

\DeclareMathOperator{\Ad}{Ad}
\DeclareMathOperator{\Ind}{Ind}
\DeclareMathOperator{\Res}{Res}

\DeclareMathOperator{\Trace}{Trace}

    \renewcommand{\H}{{\mathcal H }}
    \newcommand{\Bounded}{\mathfrak{B}}
    \newcommand{\Compact}{\mathfrak{K}}

\swapnumbers

\theoremstyle{plain}
\newtheorem{theorem}{Theorem}[section]
\newtheorem{lemma}[theorem]{Lemma}
\newtheorem{proposition}[theorem]{Proposition}
\newtheorem{corollary}[theorem]{Corollary}
\theoremstyle{definition}
\newtheorem{definition}[theorem]{Definition}
\newtheorem{example}[theorem]{Example}

\newtheorem{remark}[theorem]{Remark}

\numberwithin{equation}{section}

\title{\texorpdfstring{Parabolic induction and restriction via $C^*$-algebras and Hilbert $C^*$-modules}{Parabolic induction and restriction via C*-algebras and Hilbert C*-modules}}
\author{Pierre Clare \and Tyrone Crisp \thanks{Partially supported by the Danish National Research Foundation through the Centre for Symmetry and Deformation (DNRF92).} \and Nigel Higson \thanks{Partially supported by the US National Science Foundation DMS-1101382. }}

\begin{document}

\def\parsedate #1:20#2#3#4#5#6#7#8\empty{#4#5/#6#7/20#2#3}
\def\moddate#1{\expandafter\parsedate\pdffilemoddate{#1}\empty}
\date{}

\maketitle
 
\begin{abstract}
This paper is about the reduced group $C^*$-algebras of real reductive groups, and about Hilbert $C^*$-modules over these $C^*$-algebras.  We shall do three things.   First we shall apply theorems from the tempered representation theory of reductive groups to determine the structure of the reduced $C^*$-algebra (the result has been known for some time, but it is difficult to assemble a full treatment from the existing literature).   Second, we shall use the structure of the reduced $C^*$-algebra to determine the structure of the Hilbert $C^*$-bimodule   that represents  the functor  of parabolic induction.  Third, we shall prove that the  parabolic induction bimodule  admits a secondary inner product, using which we can define a functor  of parabolic restriction in tempered representation theory.   We shall prove in a sequel to this paper that   parabolic restriction is adjoint, on both the left and the right, to  parabolic induction in the context of tempered unitary Hilbert space representations.
\end{abstract} 

\section{Introduction}

The unitary dual $\widehat{G}$ of a locally compact group $G$ may be topologized through the uniform convergence   on compact sets of matrix coefficient functions.   The \emph{reduced dual}   is the closed subset of $\widehat G$ consisting of (equivalence classes of) irreducible unitary representations that are weakly contained in the regular representation   on $L^2(G)$.  The unitary dual identifies naturally, as a topological space,  with the spectrum of the group $C^*$-algebra $C^*(G)$, while the reduced dual identifies with the spectrum of the \emph{reduced group $C^*$-algebra} $C^*_r (G)$, which is the operator norm-closure of the $L^1$-convolution algebra of $G$ inside the algebra of bounded operators on $L^2 (G)$. For all this see    \cite{DixmierEnglish} or  \cite[Appendix F]{T}.

The purpose of this paper is to examine the structure of the reduced dual and the reduced group $C^*$-algebra in the case of a real reductive group, for which the irreducible representations in the reduced dual  are precisely Harish-Chandra's irreducible tempered representations; see for example \cite{MR946351}. We shall pay special attention to the  functor of parabolic induction, which is not surprising  given the dominant role that parabolic induction plays in constructing irreducible tempered representations.

Let $G$ be a real reductive group and let $P$ be a parabolic subgroup with Levi factor $L$ (the precise class of groups that we shall consider is described in Section~\ref{sec-red-groups}).  We shall approach parabolic induction through the $\left(C^*_r (G),C^*_r (L)\right)$-correspondence introduced in \cite{PArtmodules} (the notation used there was $\E(G/N)$; here we shall use $C^*_r (G/N)$).  The general properties of this correspondence, especially the fact that $C^*_r(G)$ acts on $C^*_r (G/N)$ by compact operators (in the Hilbert module sense), contribute in a very helpful way to the determination of the structure of $C^*_r(G)$.  In the reverse direction, once the structure of the reduced $\Cs$-algebra has been determined, it is not difficult to determine the   structure of the correspondence.

The problems of determining the unitary dual or the reduced dual as topological spaces, and of determining the structure of the associated $C^*$-algebras, have a long history.  To give just a sampling of interesting advances we offer the list \cite{MR0146681,MR0269778,MR0324429,MR0476913,MR724030,Valette85,NoteWassermann}.  The final paper in the list, a short announcement by Wassermann, gives in some sense  the final word on the subject in the case of the reduced $C^*$-algebra.  But it relies on a short announcement \cite{MR0460539} by Arthur on the structure of the Harish-Chandra Schwartz algebra of a real reductive group, and neither of Wassermann's nor Arthur's  announcements were followed by   detailed published accounts.   Because of this, and because in any case a direct approach through $C^*$-algebras, rather than Schwartz algebras, is a bit more economical (see for example the use of elementary $C^*$-algebra ideas in Lemmas~\ref{lem-easy-cstar1}, \ref{lem-two-ideals} and \ref{lem-many-ideals}), we felt it worthwhile to provide an account of the matter here.

We should emphasize the obvious, that the structure theorem for $C^*_r(G)$ relies very heavily on results from tempered representation theory, due mostly to Harish-Chandra and Langlands.  Our contribution is to indicate how the structure of $C^*_r (G)$ can be obtained as a relatively simple consequence of these results.

Actually our main interest is   parabolic induction, not   the structure of $C^*_r (G)$, and in our view the main contribution of this paper lies there. It is shown in \cite{PArtmodules}   that if $H_\tau$ is the Hilbert space of a tempered representation $\tau$ of $L$, then the Hilbert module tensor product 
\[
 C^*_r (G/N)\otimes _{C^*_r (L)} H_\tau ,
\]
which is a Hilbert space carrying a representation of $C^*_r (G)$, is the representation of $G$ parabolically induced from $\tau$.  In the final part of this paper we shall construct a functor of \emph{parabolic restriction}, from tempered representations of $G$ to tempered representations of $L$.

To do so we shall study the adjoint module $C^*_r (N\backslash G) = C^*_r (G/N)^*$.  As a vector space this is simply the complex conjugate of $C^*_r (G/N)$,  and we equip it with the structure of a $C^*_r(L)$-$C^*_r (G)$-bimodule using the formula 
\[
b\cdot \overline e \cdot a = \overline{a^*e b^*}.
\]
Our main observation is that $C^*_r (N \backslash G)$ carries a compatible $\left(C^*_r(L),C^*_r (G)\right)$-correspondence structure, namely a compatible, complete, $C^*_r (G)$-valued inner product, hence a secondary norm which is equivalent to the original one. We can then define parabolic restriction using the Hilbert module tensor product 
\[
 C^*_r (N \backslash G )\otimes _{C^*_r (G)} H_\pi ,
\]
which carries tempered representations of $G$ to tempered representations of $L$.  The details are given in Section~\ref{sec-adjoint-module}.   It is not difficult to calculate the functor, given all the structure theory in the preceding sections, and we shall give some examples in Section~\ref{sec-adjoint-module}.  But the  most important feature of the parabolic restriction functor is that it is both left and right adjoint to the functor of parabolic induction between categories of tempered unitary representations.  This we shall prove in a separate paper \cite{CCH_adjoint} (the proof is not difficult, but it involves a quite different set of ideas from operator space theory that would be a bit out of place in the present paper).  

Of course, the presence of an adjoint to parabolic induction on both sides calls to mind Bernstein's second adjoint theorem in the representation theory of $p$-adic reductive groups \cite[Chapter 3]{BernsteinPoly}.  We hope to return to the relationship between the second adjoint theorem and our bimodule elsewhere (for a few preliminary comments, see Remark~\ref{rem-bernstein} in this paper).  But at the moment, a full understanding of the parabolic restriction functor introduced here  is beyond our reach.  It would be especially interesting to obtain a geometric perspective  on parabolic restriction and the second adjoint theorem in the real case (compare \cite{BezKazh} for the $p$-adic case).

In the final section of the paper  we shall explain the relation between our parabolic restriction functor and the Plancherel formula.  Using Harish-Chandra's  wave packets we give a simple explicit formula for the $C^*_r(G)$-valued inner product on (a dense subspace of)  $C^*_r (N \backslash G )$.

\section{Compact Operators on Hilbert Modules}
\label{sec-hilb-mod}

Throughout the paper we shall use the language of Hilbert modules over $C^*$-algebras.  For   background information we refer the reader to \cite{Lance}.\footnote{Whereas the term \emph{Hilbert $C^*$-module} is used in \cite{Lance}, here we shall use the contracted form \emph{Hilbert module}.}  In this section we shall fix some   terminology and  notation, and describe some specialized ideas concerning group actions   that will  soon feature prominently. 

\subsection*{Compact Operators}

\begin{definition}
Let  $C$ be  a $C^*$-algebra, and let $\mathcal{H}_1$ and $\mathcal {H}_2$ be Hilbert $C$-modules. We shall denote by $\Bounded ( \H _1, \H _2)$ the space of bounded operators $T\colon \H _1\to \H _2$ that possess an adjoint $T^*\colon \H _2 \to \H _1$, characterized by the usual formula 
\[
\langle Tv_1 , v_2 \rangle = \langle v_1 , T^* v_2 \rangle
\]
for all $v_1\in \H _1$ and all $ v_2 \in \H  _2$.  We say that $T$ is \emph{adjointable} if it possesses an adjoint.
\end{definition}

\begin{definition}
\label{def-cpt-op}
An adjointable operator in $\Bounded (\H_1,\H_2)$ is \emph{compact} if it lies in the operator norm-closure of the linear  span of   the  \emph{elementary operators} 
\[
 v_2\otimes v_1 ^* \colon v \mapsto v_2 \langle v_1, v\rangle
 \]
 determined by vectors $v_1 \in \H _1$ and $v_2 \in \H _2$. 
We shall denote  by $\Compact ( \H _1 ,\mathcal  H_2)$ the closed subspace in $\mathcal{B}(\mathcal{H}_1 , \mathcal {H}_2)$  consisting of all compact operators. 
\end{definition}

 See \cite[Chapter 1]{Lance}.  The algebra $\Compact(\H)$  of all compact operators on a single Hilbert $C$-module $\H$ is in fact a $C^*$-algebra (so is the algebra of all bounded, adjointable operators, but this will play a lesser role).  
 
\subsection*{Group Actions}
Let $\H$ be a Hilbert module over a $C^*$-algebra $C$.  We shall need to consider   group actions on the $C^*$-algebra $\Compact (\H)$ that are constructed from the following sorts of automorphisms of $\H$:

\begin{definition}
\label{def-twisted-action}
 By a  \emph{twisted unitary automorphism}  of $\H$  we shall mean the following data:
\begin{enumerate}[\rm (a)]
\item an   automorphism  $c\mapsto \alpha (c)$  of the $C^*$-algebra $C$, and  
\item a   complex-linear automorphism 
$
v \mapsto U(v)
$
 of $\H $,   with the property  that 
\begin{enumerate}[\rm (i)]
\item  $ U(vc )  = U(v) \alpha (c)    $ for all   $v\in \H $, all $c\in C$, and 
\item $\langle U(v_1), U(v_2)\rangle = \alpha(\langle v_1,v_2\rangle)$ for all $v_1,v_2\in \H $.
\end{enumerate}
\end{enumerate} 
We shall  say that the automorphism $U\colon \H\to \H $  \emph{covers} the $C^*$-algebra automorphism $\alpha\colon C \to C$.
\end{definition}

\begin{example}
\label{ex-vector-bundle}
 Let $X$ be a locally compact  space, and let 
 \[
 w\colon X \longrightarrow  X\]
  be a homeomorphism.   Let  $E$ be an equivariant Hermitian vector bundle over $X$, and let 
  \[
  \tilde w \colon E \longrightarrow  E
  \]
   be a homeomorphism that covers $w$ and is fiberwise a unitary vector space isomorphism.   Let $C = C_0(X)$ and let  $\H $ be the Hilbert $C$-module of continuous sections of $E$ vanishing at infinity on $X$. Then the formulas
\[
\alpha_w(f)(x) = f(w^{-1}x)\quad \text{and}\quad U_{\tilde w}(v)(x) =  \tilde w(v(w^{-1}x)) \in E_x
\]
define a  twisted unitary automorphism of $\H$.  
\end{example}

 \begin{definition}
\label{def-inner-auto}
 Suppose that $u$ is a unitary element in $C$, or in the multiplier  algebra of $C$; see \cite[Section 3.12]{Pedersen} or \cite[Chapter 2]{Lance}.  The multiplier $C^*$-algebra  contains $C$ as a closed, two-sided ideal, and the right $C$-action on $\H$ extends uniquely to the multiplier algebra. Because of this, we can define a twisted unitary automorphism of $\H$ as follows:
 \begin{enumerate}[\rm (a)]
\item $\alpha(c) = ucu^*$ for all $c\in C$,  and 
\item $U(v) = vu^*$ for all $v\in \mathcal{H}$.
\end{enumerate}
We shall call $U$    an \emph{inner} twisted unitary automorphism of $\H$. 
\end{definition}

Every twisted unitary automorphism  $g=(\alpha, U)$ of a Hilbert  module $\H $  induces an ordinary $C^*$-algebra automorphism \[
\operatorname{Ad}_g\colon \Compact  (\H )\longrightarrow \Compact (\H),
\]
 since if an operator $T\colon \H\to \H$    is compact in the sense of Definition~\ref{def-cpt-op}, then so is the composition 
\begin{equation}
\label{eq-action-on-ops}
\operatorname{Ad}_g(T)  =  U\circ T\circ U^{-1} \colon \H \longrightarrow \H .
 \end{equation}
   In fact
\begin{equation}
\label{eq-action-on-ops2}
 g( v_2\otimes v_1 ^*) =  U(v_2) \otimes    U(v_1) ^* .
 \end{equation}

\begin{lemma}
If $g$ is an inner automorphism of $\H$, then the induced automorphism $\operatorname{Ad}_g$  of $\Compact (\H)$ is the identity automorphism. \qed
\end{lemma}
 
Assume now that a group $ W$ acts by   automorphisms\footnote{More generally we could consider   an action by \emph{outer} automorphisms, that is, an action modulo inner automorphisms.  But this extra generality will not be needed for the examples we shall consider in this paper.}  on  the $C^*$-algebra $C$.  Assume   that associated to each element $w\in W$ we are given  a twisted unitary automorphism
\[
U_{w} \colon \H \longrightarrow \H
\]
that covers $w$.  In addition, assume that for all $w,z\in W$ the composition 
\[
\H \stackrel{ U_{w}} \longrightarrow \H \stackrel {U_{z}} \longrightarrow \H
\]
is the composition  of   $U_{zw}$ with   an inner automorphism (on either the left or the right).  Necessarily\footnote{This would not be so if we were considering outer actions; see the previous footnote.} the inner automorphism covers the \emph{identity} automorphism of $C$; that is, it is constructed, as in Definition~\ref{def-inner-auto}, from a \emph{central} unitary $u$ in $C$  or in the multiplier algebra of $C$. 

\begin{definition}
\label{def-proj-act}
In this situation we shall say that the group $W$ acts \emph{projectively} on the Hilbert module $\H$.
\end{definition}

Such a projective action  of $  W$ on $\H$  induces an ordinary action  of $W$ by $C^*$-algebra automorphisms on $\Compact(\mathcal{H})$, and we shall be studying throughout the paper the associated fixed-point algebra:

\begin{definition}
We shall denote by $\Compact  (\H )^W$ the fixed-point $C^*$-subalgebra of $\Compact (\H)$ under the above action.
\end{definition}

\begin{remark}
Given any projective action  we can form the  group extension
\begin{equation}
\label{eq-proj-extension}
1\longrightarrow \operatorname{Inn}(\H) \longrightarrow \widetilde W \longrightarrow W \longrightarrow 1 ,
\end{equation}
in which $ \operatorname{Inn}(\H)$ is the  group of inner automorphisms of $\H$ associated to central unitary elements of $C$ or its multiplier algebra, while  
$
\widetilde W 
$
is the Cartesian product $\operatorname{Inn}(\H) \times W$ as a set, but with group structure
\[
(U_1,w_1) (U_2,w_2) = ( U_{12}, w_1w_2)
\]
where
\[
U_1 U_{w_1}U_2 U_{w_2} = U_{12} U_{w_1w_2} \colon \H \longrightarrow \H
.
\]
There is then an actual (rather than projective) action of  $\widetilde W$   on $\H$ by twisted unitary automorphisms.  The fixed point algebras associated to $\widetilde W$ and $W$ are the same.
 \end{remark}

  \subsection*{Hilbert Correspondences and Tensor Products}
  

\begin{definition}\label{def-A-B-bimodule}
Let $B$ and $C$ be $C^*$-algebras and let $\H$ be a Hilbert $C$-module. We shall call $\H$ a \emph{correspondence from $B$ to $C$}, or a \emph{$(B,C)$-correspondence} if it is equipped with an \emph{action  homomorphism} of $C^*$-algebras
\begin{equation}
\label{eq-action-map}
B  \longrightarrow \Bounded  (\H ).
\end{equation}
\end{definition}

\begin{definition}
\label{def-interior-tp}
Let $\mathcal E$ be a Hilbert  $B$-module and let $\H$ be a  {correspondence from $B$ to $C$}. The \emph{interior tensor product}  $
 \E\otimes _B \mathcal {H}$ is constructed from the algebraic tensor product $\E\otimes^{\text{alg}}_B \H$ by completion in the norm associated to the $C$-valued inner product
 \begin{equation*}
 \bigl \langle   e_1 \otimes v_1 , e_2 \otimes v_2 \bigr \rangle_C  =
\bigl  \langle v_1, \langle e_1,e_2\rangle_B    v_2\bigr \rangle_C  .
\end{equation*}
On the right hand side, the element $\langle e_1,e_2\rangle_B\in B$ is regarded as an operator on $\mathcal{H}$ via the action homomorphism (\ref{eq-action-map}).   
\end{definition}

See \cite[Proposition 4.5]{Lance} for details.  If $\E$ is a Hilbert correspondence from $A$ to $B$, then the interior tensor product is a correspondence from $A$ to $C$.
 
 \begin{lemma}
\label{lem-cpt-tensor-one}
If $\E$ is a Hilbert $B$-module,   if $B$ acts on a Hilbert $C$-module $\H$ through compact operators, and if $S$ is a compact operator on $\E$, then $S\otimes I$ is a compact operator on $\E\otimes _B \H$.  \end{lemma}

\begin{proof}
This follows directly from \cite[Proposition 4.7]{Lance}.
\end{proof}

For the rest of this section we shall be concerned with    the situation  in which a finite group $W$ acts projectively on a Hilbert $C$-module $\H $ by twisted unitary automorphisms, as in the previous section. We shall take 
\begin{equation}
\label{eq-def-of-B}
B = \Compact (\H)^W,
\end{equation}
which acts on $\H$ in the obvious way and gives $\H$ the structure of a correspondence from $B$ to $C$.

We shall examine the structure of the interior tensor product $\E \otimes _B \H$ in this case, and the structure of the $C^*$-algebra $\Compact (\E\otimes _B \H)$ of compact operators on the tensor product. 

\begin{lemma}
\label{lem-W-action-def}
If $\H$ and $B$ are as above, then the  formula  
\[
U_w (e\otimes v)  = e \otimes U_w(v)
\]
 defines a projective action of $W$ by twisted unitary automorphisms on the interior tensor product $\E\otimes _B \H$.
\end{lemma}
 
 \begin{proof}
The action on the algebraic tensor product is isometric in the interior tensor product norm, and so extends to the completion $\E\otimes _B \H$.  The compatibility conditions in Definition~\ref{def-twisted-action} 
can be checked on the algebraic tensor product, and then they extend by continuity to the completion.
 \end{proof}
  
Now let  $e\in \E$. The formula
 \begin{equation}
 \label{eq-T-e}
T_e \colon v\mapsto e \otimes v
 \end{equation}
defines an operator $T_e \colon \H \to \E\otimes _B \H$ with adjoint 
 \begin{equation}
 \label{eq-T-e-star}
T_e^* \colon f\otimes v \mapsto \langle e,f\rangle v.
 \end{equation}
 
 \begin{lemma}
Each of the operators $T_e$ in \textup{(\ref{eq-T-e})} is compact and $W$-equivariant. 
 \end{lemma}
 
 \begin{proof}
 The $W$-equivariance of $T_e$ is clear from the definition of the action in Lemma~\ref{lem-W-action-def}.
Assuming, as we are in (\ref{eq-def-of-B}), that   $B$ acts on $\H$ through compact operators,   the composition 
\[
T^*_e T_e \colon v \mapsto \langle e, e\rangle v 
\]
is evidently compact. Therefore $T_e$ is compact, too. \end{proof}

Now consider the map
\begin{equation}
\label{eq-module-action-map}
\mathcal E \longrightarrow  \Compact  (\H , \E\otimes _B \H )^W
\end{equation}
 that sends $e\in \E$ to the compact operator $T_e$.  Consider the target space  as a Hilbert $B$-module under the inner product
 \[
 \langle S , T\rangle = S^* T\in \Compact (\H)^W.
 \]

\begin{proposition}
\label{prop-module-structure}
The  map
\textup{(\ref{eq-module-action-map})}
is an isometric isomorphism of Hilbert $B$-modules.
\end{proposition}

\begin{proof}
Consider the diagram of  isometric  isomorphisms of Hilbert $B$-modules
\[
\E \stackrel \cong \longleftarrow \E \otimes _B B \stackrel\cong \longrightarrow \E\otimes _B  \Compact (\mathcal{H})^W ,
\]
in which the left-hand map is multiplication and the right-hand map simply recalls the definition of $B$ in (\ref{eq-def-of-B}).  Using the isomorphisms, we can think of (\ref{eq-module-action-map}) as a map 
\[
 \E\otimes _B  \Compact (\mathcal{H})^W  \longrightarrow  \Compact  (\H , \E\otimes _B \H )^W .
\]
On elementary tensors the map has the form
\[
e\otimes (v_1 \otimes v_2^*) \mapsto (e\otimes v_1) \otimes v_2 ^*,
\]
It preserves inner products and has dense range, so it is an isometric isomorphism.
\end{proof}

We turn now to a description of $\Compact  (\E)$.  The formula   $S\mapsto S \otimes I$ defines a  homomorphism of $C^*$-algebras
\begin{equation}
\label{eq-cpt-iso1}
\Compact  (\E) \longrightarrow \Bounded (\E\otimes _B \H ) .
\end{equation}
See \cite[p.42]{Lance}.

\begin{lemma}
The operators $S\otimes I $ are compact and $W$-invariant.
\end{lemma}

\begin{proof}
The $W$-invariance is clear. Compactness is a consequence of Lemma~\ref{lem-cpt-tensor-one}. \end{proof}

\begin{proposition}
 \label{prop-cpt-iso2}
The map $S\mapsto S\otimes I$ determines  an isomorphism of $C^*$-algebras
 \begin{equation*}
\Compact  (\E)\stackrel \cong  \longrightarrow \Compact  (\E\otimes _B \H )^W .
\end{equation*}
\end{proposition}

\begin{remark}
The value of this result is that in our application the $C^*$-algebra $C$ will have a very simple structure---in fact it will be abelian, and it will be easy to calculate the tensor product $\E\otimes _B \H$.
\end{remark}

\begin{proof}[Proof of the Proposition]
It is proved in  \cite[Proposition 4.7]{Lance} that the homomorphism is injective, and moreover it is proved there that the homomorphism is an isomorphism when $W$ is trivial.  The following small modification of the argument in \cite{Lance} handles surjectivity in the general case.  It suffices to show that the operator
\begin{equation}
\label{eq-avgd-op}
\operatorname{Average}  \bigl [ (e_2\otimes v_2 ) \otimes (e_1 \otimes v_1)^* \bigr ]  \in  \Compact  (\E\otimes _B \H )^W
\end{equation}
lie in the image of our homomorphism of $C^*$-algebras, where the average is taken  over the $W$-action on compact operators.  The operator (\ref{eq-avgd-op}) acts on $\E\otimes _B \H$ as follows:
\begin{equation}
\label{eq-average-action}
e \otimes v 
	\mapsto 
\frac{1}{|W|} \sum _{w\in W}  e_2 \otimes U_w(v_2) \, \bigl \langle U_w(v_1) , \langle e_1, e\rangle \, v \bigr \rangle  .
\end{equation}
Here, in the case of a projective action, we lift each element $w\in W$ to an element of the group $\widetilde W$ in (\ref{eq-proj-extension}) before acting on $\H$; the choice of lift does not affect the formula.  Using (\ref{eq-action-on-ops2}) we can rewrite (\ref{eq-average-action}) as 
\begin{equation*}
\label{eq-average-action2}
e \otimes v 
	\mapsto 
 e_2 \otimes  \operatorname{Average} [v_2\otimes v_1^*]  \cdot  \langle e_1, e\rangle \cdot v  ,
\end{equation*}
or equivalently (since the tensor product is over $B=\Compact (\H)^W$)
\begin{equation*}
\label{eq-average-action3}
e \otimes v 
	\mapsto 
 e_2 \cdot  \operatorname{Average} [v_2\otimes v_1^*]  \otimes   \langle e_1, e\rangle\cdot  v  .
\end{equation*}
But this is the formula for the action of the operator $S\otimes I$, where 
\[
S =  e_2  \operatorname{Average} [v_2\otimes v_1^*] \otimes e_1^* \in \Compact (\E) ,
\]
and so the proof is complete.
\end{proof}

\subsection*{Direct Sums}

In the coming sections the $C^*$-algebras $B$ of concern to us, namely the reduced group $C^*$-algebras of reductive groups, will  be \emph{direct sums} of $C^*$-algebras of the type $\Compact(\H)^W$ considered in the last section.  Thus we shall be considering $C^*$-algebras of the form 
\begin{equation}
\label{eq-dir-sum-decomp-B}
B = \bigoplus_\alpha B_\alpha = \bigoplus_\alpha \Compact(\H_\alpha)^{W_\alpha} .
\end{equation}
The direct sum is to be taken in the $C^*$-algebraic sense, which is to say that $B$ is the completion of the algebraic direct sum in the supremum norm.  The $C^*$-algebra operations are defined coordinatewise.

If $\E$ is a Hilbert  module over $B=\oplus _\alpha B_\alpha$, then $\E$ decomposes in a unique way as a direct sum
\begin{equation}
\label{eq-dir-sum-decomp-E1}
\E = \bigoplus _\alpha \E_\alpha,
\end{equation}
with $\E_\alpha$ a Hilbert module over $B_\alpha$.  Once again, the direct sum here is the completion of the algebraic direct sum in the supremum norm, and all operations are defined pointwise (so for example distinct summands of $\E$ are orthogonal to one another, and the inner product of two elements from a single summand $\E_\alpha$ lies in the summand $B_\alpha$ of $B$).

In the situation displayed in (\ref{eq-dir-sum-decomp-B}), applying Proposition~\ref{prop-module-structure}  coordinatewise we obtain an isomorphism
\begin{equation}
\label{eq-dir-sum-decomp-E2}
\mathcal E \stackrel\cong \longrightarrow  
\bigoplus_\alpha \Compact  (\H_\alpha , \E\otimes _{B } \H_\alpha )^{W_\alpha},
\end{equation}
which describes the summands $\E_\alpha$.  Here $\H_\alpha$ is given a left action of $B$ through the projection
\[
B \longrightarrow B_\alpha =  \Compact(\H_\alpha)^{W_\alpha} .
\]
Similarly, applying Proposition~\ref{prop-cpt-iso2} coordinatewise we obtain an isomorphism 
\begin{equation}
\label{eq-dir-sum-decomp-E3}
\Compact  (\E)\stackrel \cong  \longrightarrow \bigoplus_\alpha  \Compact  (\E_\alpha \otimes _{B} \H_\alpha )^{W_\alpha} .
\end{equation}

 \section{Reductive Groups and Parabolic Subgroups}
 \label{sec-red-groups}

 We shall not attempt to strive for the utmost generality in the class of groups we shall consider. Instead we shall aim for (relative) simplicity. This will  also guarantee that the diverse references that we shall cite in the next several sections will actually  cover our class of groups. 
 
Consider first the class of   \emph{connected, self-adjoint matrix groups}.  This is the class of those  closed, connected  subgroups of the matrix groups $\mathrm{GL}(n,\R)$ that are invariant under the transpose operation on matrices.  Given such a group $G\subseteq GL(n,\R)$, the connected Lie subgroup  $G_\C\subseteq \mathrm{GL}(n,\C)$ whose Lie algebra is the complexification of the Lie algebra of $G$  is a connected (in the algebraic sense \cite[\S 7.3]{MR0396773}) reductive algebraic group defined over $\R$.  The group $G$ is an open subgroup (in the usual analytic topology on matrices) in the group of real points of this algebraic group.   The group of all real points need not itself be connected, although it has at most finitely many components.

Although connectedness is a natural assumption, for technical reasons it will be more convenient to work with the full group of real points.  So from now on, we shall let $G \subseteq \mathrm{GL}(n,\R)$ be a self-adjoint group which  is also  the group of   real points of a connected (and necessarily reductive) algebraic group defined over $\R$.  For brevity, we shall simply say that $G$ is  a \emph{real reductive group}.  

This convention excludes some examples that we might otherwise consider (for instance the group of all matrices with positive determinant), but  it has the winning advantage for us that the class of groups under consideration now  is closed under passage to the  standard Levi subgroups whose definition we shall recall in a moment.

\begin{definition}
\label{def-K-grp}
Let $K = G\cap O(n)$, which is  a maximal compact subgroup of $G$ \cite[Proposition 1.2]{Knapp1}.
\end{definition}

\begin{definition}
\label{def-A-grp}
 Let $A$ be a maximal abelian subgroup  of positive-definite matrices in $G$. \end{definition}

\begin{remark}
The group $A$ is not unique, but it is unique  up to conjugacy by an element of $K$ \cite[Chap. VII]{KnappBeyond}.   
\end{remark}

The positive-definite group $A$ is isomorphic to its Lie algebra $\mathfrak a$ via the exponential map.  We can use  elements of $\mathfrak a$ to define \emph{standard Levi subgroups} of $G$, as follows. 

\begin{definition}
The \emph{standard Levi subgroups} of $G$ (for a given choice of subgroup $A\subseteq G$) are the subgroups of the form 
\[
L = L_X= \{\, g \in G : \exp(tX) g \exp(-tX) = g \,\, \forall t \in \R \, \},
\]
associated to elements $X\in \mathfrak a$.  They are real reductive groups.
\end{definition}

There is a  dense open set of elements   $X\in \mathfrak a$ that all define the same group $L$.  This particular $L$  is minimal in dimension among all standard Levi subgroups and indeed is contained in every other standard Levi subgroup (for example, if $G=\mathrm{GL}(n,\R)$, and if $A$ is the group of positive diagonal matrices, then all $X$ with distinct diagonal entries define the same standard Levi subgroup  of diagonal matrices).   Fix a connected component  of this dense open set in $\mathfrak a$, and call it the \emph{positive chamber} $\mathfrak a_{+}\subseteq \mathfrak a$  (there are finitely many choices).

\begin{definition}
The \emph{standard unipotent subgroups} of $G$ (for a given choice positive-definite group $A$ and positive  chamber $\mathfrak a_{+}\subseteq \mathfrak a$) 
are the closed subgroups 
\[
N = N_X =  \{\, g \in G :  \lim_{t \to + \infty} \exp(tX) g \exp(-tX) = e   \, \},
\]
associated to elements $X\in \overline{\mathfrak a_{+}}$.  The \emph{standard parabolic subgroups} of $G$ are the closed subgroups  
\[P =P_X = L_XN_X=LN ,\]
associated to elements $X\in \overline{\mathfrak a_{+}}$.  
\end{definition}

It is obvious from the definitions that $L$ normalizes $N$, so the product $P= LN$ is indeed a subgroup, isomorphic to the semidirect product of $L$ acting  on $N$ by conjugation.  In fact $P$  is a \emph{closed} subgroup of $G$, diffeomorphic to the Cartesian product of the spaces $L$ and $N$  \cite[Chap. VII \S7]{KnappBeyond}.

\begin{example} If $G=\mathrm{GL}(n,\R)$, and if $\mathfrak a _{+}$ consists of diagonal matrices whose entries increase down the diagonal, then the standard parabolic subgroups are the various block upper triangular subgroups (for the various possible sequences of block sizes).  Their \emph{Levi factors} $L$ are the block diagonal groups, and the unipotent subgroups $N$ are the block unipotent upper triangular subgroups.
 \end{example}
 
\begin{remark}
Different choices of chamber $\mathfrak a_{+}\subseteq \mathfrak a$ are conjugate to one another via elements of $K$ that normalize $\mathfrak a$.  As a result, different choices of chamber lead to conjugate families of standard parabolic subgroups.
\end{remark}

\section{Parabolic Induction}
\label{sec-parabolic-induction}

Fix throughout this section  a real reductive group $G$ (along with a choice of positive-definite subgroup $A$ and positive chamber $\mathfrak a_{+}$).  In addition, fix  a standard parabolic subgroup $P=LN \subseteq G$, as in the previous section. 

Apart from being a subgroup of $P=L\ltimes N$, the  {Levi factor} $L=P/N$ is also a quotient. So if $\tau\colon L \to \mathrm{U}(H)$ is a unitary representation of $L$, then we can consider $\tau$ as a representation of $P$ too, and so form the unitarily induced representation 
\[
\Ind_P^G \tau \colon G \longrightarrow \mathrm{U}(\Ind _P^G H) .
\]
This is the functor of \emph{parabolic induction}, going from unitary representations of $L$ to unitary representations of $G$, and its behaviour on tempered unitary representations   will be our main concern in the rest of the paper.

 We shall now  recall the construction of the $\left(C^*_r(G),C^*_r (L)\right)$-correspondence $C^*_r (G/N)$ from \cite[Section 2]{PArtmodules}, and then prove a few elementary facts about it.

  As a Banach space,  $C^*_r (G/N)$ is a   completion of the space of smooth, compactly supported functions on the homogeneous space $G/N$.  There is a $G$-invariant smooth measure on $G/N$, which is unique up to a multiplicative constant. We choose it with respect to the fixed Haar measures on $G$ and $N$ so that \[\int_Gf(g)\,dg=\int_{G/N}\int_N f(gn)\,dn\,d(gN)\] for any measurable function $f$.  By $G$-invariance the  natural left translation action of $G$ on $C^\infty _c(G/N)$ is unitary for the $L^2$-inner product. There is an associated convolution action 
  \[
  C_c^\infty (G) \otimes C_c^\infty (G/N) \longrightarrow C_c^\infty(G/N)
  \]
  defined by the usual formula
  \begin{equation}\label{eq-def-convol}
  (f_0 * f)(gN) = \int _G f_0(\gamma)f(\gamma^{-1}gN)\, d\gamma.
  \end{equation}

  The $G$-invariant measure on $G/N$ is    not  invariant for the natural \emph{right} action of the Levi factor $L$. Instead there is a  character  $\delta \colon L  \to \R^\times_{+}$ such that  
\[
\int_{G/N} f(x \ell  ) \, dx = \delta (\ell)^{-1} \int _{G/N} f(x)\, dx
\]
for all $f\in C_c^\infty(G/N)$ and all $\ell \in L$.  In fact 
\begin{equation}
\label{eq-delta-def}
\delta (\ell)    =  \left |  \operatorname{det} \left ( \operatorname{Ad}_\ell \colon \mathfrak n \to \mathfrak n\right ) \right |  .
\end{equation}
But if we adjust the right   action    of $L$  on the function space $C_c^\infty(G/N)$  by means of the formula
\[
(f\cdot \ell )(x)  = \delta (\ell )^{- \frac 12} f(x\ell^{-1}) ,
\]
then we obtain  a unitary action for the natural $L^2$-inner product. There is an associated convolution action
\[
C_c ^\infty (G/N) \otimes C_c^\infty (L) \longrightarrow C_c^\infty (G/N)
\]
defined by
\[
\begin{aligned}
( f  * f _1) ( x)  &  = \int _L \delta (\ell) ^{-\frac 12} f(x\ell^{-1}) f_1 (\ell ) \, d\ell  \\
	& =  \int _L \delta (\ell) ^{\frac 12} f(x\ell) f_1 (\ell^{-1}) \, d\ell,  \\
	\end{aligned}
\]
where the integrals are equal because $L$ is unimodular. Finally, a $C_c^\infty (L)$-valued inner product is defined on $C_c^\infty (G/N)$ by
\[
 \langle h, f \rangle \colon \ell \mapsto  \delta (\ell)^{\frac 12}  \int _{G/N}\overline{h (x)}f (x\ell) \, dx ,
\]
or equivalently
\[
 \langle h, f  \rangle \colon \ell \mapsto \delta (\ell)^{-\frac 12} \int _{G/N} \overline{h(x\ell^{-1})}f (x ) \, dx 
\]
(the integrands are compactly supported functions because the right action of $L$ on the homogeneous space $G/N$ is proper, which in turn follows from the fact that $LN$ is a closed subgroup of $G$). All these structures extend by completion to give the Hilbert module $C^*_r (G/N)$.  See \cite[Proposition 1]{PArtmodules}.

\begin{remark}
In \cite{PArtmodules}, the Hilbert module $C^*_r (G/N)$ is shown to admit a left action of the \emph{full} group $C^*$-algebra $C^*(G)$.  To see that the left action  factors through $\Csr(G)$, let $f_0\in C_c^\infty(G)$ and let 
\[
T \colon C^*_r (G/N) \longrightarrow C^*_r (G/N)
\]
 be the convolution operator determined  by the formula  \eqref{eq-def-convol}.  We need to prove that the operator norm of $T$ is bounded by the reduced $C^*$-algebra norm of $f_0$. 
 
Let $\psi$ be a  faithful state of $\Csr(L)$.  The formula 
\[
\langle f_1,f_2\rangle_\psi = \psi \left ( \langle f_1,f_2 \rangle_{C^*_r (G/N)} \right )
\]
defines a scalar inner product on $C^*_r (G/N)$.  Denote by $\Csr(G/N)_\psi$   the associated Hilbert space completion.  
Each bounded, adjointable operator on $C^*_r (G/N)$ extends to a bounded operator on $C^*_r (G/N)_\psi$, and the \emph{localization map}
 \[
 \Bounded(\Csr(G/N))\longrightarrow\Bounded(\Csr(G/N)_\psi)
 \]
 defined in this manner  is an injective, and hence isometric, homomorphism of $C^*$-algebras; see \cite[page 55]{Lance}. 
 
Returning to the matter at hand, it follows that  the norms of $T$ as an operator on $\Csr(G/N)$ and on $\Csr(G/N)_\psi$ are equal. But the representation of $C^*(G)$ on $C^*_r (G/N)_\psi$ is easily checked to be weakly contained in $L^2 (G/N)$, and the representation of $C^*(G)$ on this Hilbert space factors through $C^*_r (G)$ because $N$ is amenable. 
\end{remark}

The significance of the correspondence $C^*_r (G/N)$ is that it implements the functor of parabolic induction:

\begin{proposition}\textup{(See \cite[Corollary 1]{PArtmodules}.)}\label{prop-Cups-induce}
Let $\tau$ be a tempered unitary  representation of $L$ on a Hilbert space $H$.   The parabolically induced representation $\Ind_P^G \tau$ is unitarily equivalent to the representation of $G$ on the Hilbert space $C^*_r (G/N)\otimes _{C^*_r (L)} H$.
\end{proposition}

\begin{remark} 
We might call the correspondence $\Csr(G/N)$ the \emph{$\Cs$-algebraic universal principal series}, following  similar terminology that is used in the $p$-adic context; see \cite{BezKazh}. 
\end{remark}

\begin{definition} Let $\H$ be a correspondence from $C^*_r(L)$ to a $C^*$-algebra $C$. We define the \emph{parabolically induced Hilbert module} $\Ind _P^G \H$  to be the Hilbert module 
\[
\Ind_P^G \H = C^*_r (G/N)\otimes_{C^*_r (L)} \H.
\]
It is a correspondence from $C^*_r(G)$ to $C$.
\end{definition}

\begin{proposition}
\label{prop-compact-action}The \Cs-algebra $\Csr(G)$ acts  by compact operators on the Hilbert module $C^*_r (G/N)$.\end{proposition}

\begin{proof}  Let  $f_0\in C_c^\infty (G)$ and let  $f\in C_c^\infty (G/N)$.  Then 
\[
\begin{aligned}
(f_0*f)(x) & = \int _G f_0(\gamma) f(\gamma^{-1}x) \, d\gamma \\
& = \int _{G/N} k(x,y) f(y) \, dy ,
\end{aligned}
\]
where
\[
k(g_1N,g_2N) = \int _N f_0(g_1ng_2^{-1}) \, dn .
\]
The kernel function 
\begin{equation}
\label{eq-kernel-G-mod-N}
k\colon G/N \times G/N \longrightarrow  \C
\end{equation}
 defined from $f_0$ by the above integral  is $\delta$-homogeneous   under the right action of $L$  in the sense that 
\begin{equation}
\label{eq-delta-homog}
k(x\ell, y\ell) = \delta (\ell)^{-1} k(x,y) .
\end{equation}
Here $\delta$ is the character (\ref{eq-delta-def}). So the  support of $k$ is an $L$-invariant closed set in $G/N\times G/N$ for the diagonal right action of $L$. 

We claim that the image of the support of the kernel function $k$  in the quotient space $\left ( G/N\times G/N\right ) / L$  is  a \emph{compact} set. To see this, consider the mapping from $G/P\times \operatorname{supp}(f_0)$ to $\left( G/N\times G/N \right) / L$ given by
\begin{equation}\label{eq-G-mod-P-supp}
(h_1 P, h_2)\mapsto (h_1 N, h_2^{-1} h_1 N)L.
\end{equation}
If $(g_1 N, g_2 N)$ lies in the support of $k$ then there is some $n\in N$ for which $g_1 n g_2^{-1}$ lies in the support of $f_0$, and then the image of $(g_1 N, g_2 N)$ in the quotient $\left( G/N \times G/N \right) / L$ is equal to the image of the point $(g_1 P, g_1 n g_2^{-1})$ under \eqref{eq-G-mod-P-supp}. Thus the image of the support of $k$ in $\left( G/N \times G/N \right) / L$ is contained in the image of \eqref{eq-G-mod-P-supp}. But $G=KP$  \cite[Proposition 7.83]{KnappBeyond}, so $G/P$ is compact and \eqref{eq-G-mod-P-supp} has compact image.

Now, any smooth kernel function as in \eqref{eq-kernel-G-mod-N} with the   homogeneity   property \eqref{eq-delta-homog}  whose support is compact   in $\left ( G/N\times G/N\right ) / L$ defines a bounded, adjointable operator on $C^*_r(G/N)$.  We shall show that these operators are all compact. 

Each such kernel function $k$ may be written in the form 
 \[k(x,y)= \int_L u(x\ell,y\ell)\delta(\ell)\,d\ell\] for some  $u\in C_c^\infty(G/N\times G/N)$. The function  $u$  may  be approximated in the uniform norm by linear combinations of elementary functions  $(x,y)\mapsto h_1(x)h_2(y)$, with all the elementary functions uniformly compactly supported. It suffices to show that  a kernel function  
\[
 k(x_1,x_2) = \int _L h_1 (x_1\ell) h_2(x_2\ell) \delta (\ell) \, d\ell 
 \]
associated to a single elementary function gives rise to a compact operator.  

The operator in this case maps $f\in C_c^\infty (G/N) $ to the function 
 \[
 x\mapsto \int_{G/N} \int _L h_1(x\ell) h_2(y\ell)\delta(\ell) f(y) \, dy d\ell ,
 \]
 which may be re-written as 
 \[
  x\mapsto \int _L \delta(\ell)^{\frac 12} h_1(x\ell) \int_{G/N}  h_2(y\ell)\delta(\ell)^{\frac 12} f(y) \, dy d\ell .
 \]
This is precisely how the rank one operator
 \[
 h_1 \otimes h_2^*\colon  f \mapsto h_1 \langle h_2 , f\rangle ,
 \]
acts.    
\end{proof}

\begin{corollary}
\label{cor-cpt-action1}
If $C^*_r (L)$ acts  through compact operators on a Hilbert module $\H$, then $C^*_r (G)$ acts through compact operators on the parabolically induced Hilbert module $\Ind _P^G  \H$.
\end{corollary}

\begin{proof}
This is a consequence of Proposition~\ref{prop-compact-action} and Lemma~\ref{lem-cpt-tensor-one}.
\end{proof}

The corollary applies to any irreducible unitary Hilbert space  representation of $G$  as a result of a fundamental theorem of Harish-Chandra (see \cite[Theorem 6, p.230]{HCRSS1}):  

\begin{theorem}\label{thm-cpt-action1}
The $C^*$-algebra of a real reductive group acts by compact operators in any irreducible unitary representation of $G$.
\end{theorem}

\begin{remark} For proofs of this theorem more congenial to operator algebra theory, see \cite[Theorem 2]{Godementatheory} or \cite{MR0104756}. See also \cite{MR0099380}.
\end{remark}

Let us now examine the tempered representations of $L$ in more detail.  The group $L$ factors canonically as a Cartesian product of two closed and commuting Lie subgroups. We'll follow tradition and express this in terms of the \emph{Langlands decomposition}
\[
P = \left(M_P \times A_P\right)\ltimes N_P = M_P A_P N_P,
\]
where:
\begin{enumerate}[\rm (a)]
\item  $N_P = N$.
\item  $M_P A_P = L$.  
\item $A_P$ is the group of positive-definite matrices in the center of $L$.
\item $M_P$ may be characterized as the subgroup of $L$ generated  by the compact subgroups of $L$.  We shall call it the \emph{compactly generated part of $L$}.
\end{enumerate}
See for example \cite[Chap. VII \S7]{KnappBeyond}. We  find that every tempered irreducible representation of $L$ is a product  $\sigma\otimes \varphi$ of a tempered irreducible representation $\sigma$ of $M_P$ with a unitary character $\varphi $ of $A_P$.
We are especially interested in the special case where
\[
\sigma\colon M_P \longrightarrow \mathrm{U}(H_\sigma)
\]
is an irreducible \emph{square-integrable} unitary Hilbert space representation of $M_P$. It is not actually necessary to specialize to square integrable representations for the results of this section, but we shall do so anyway, to fix ideas and notation for the next two sections.   The representations obtained from such $\sigma$ by extending $\sigma\otimes\varphi$ to $P$ trivially across $N$ then inducing to $G$: \[\Ind_P^G (\sigma \otimes \varphi)\] are the (unitary) \emph{principal series} representations of $G$.

\begin{definition}
Denote by 
\[
\H_\sigma = C_0(\widehat A, H_\sigma)
\]
 the Hilbert $C_0(\widehat A)$-module of continuous functions, vanishing at infinity, from the Pontrjagin dual $\widehat A$ into the Hilbert space $H_\sigma$.  The $C_0(\widehat A)$-module action is by pointwise multiplication, and the $C_0(\widehat A)$-valued inner product is the pointwise inner product of Hilbert space-valued functions on $\widehat A$.
\end{definition}

The $C^*$-algebra $C^*_r (L)$ acts on $\H_\sigma$ through the family of representations $\sigma \otimes \varphi$  of $L= MA$ on the Hilbert space $H_\sigma$.  Thus   if $f\in C^*_r (L)$ and $h\in  C_0(\widehat A, H_\sigma)$, then 
\[
(f \cdot h) (\varphi) = (\sigma \otimes \varphi) (f) h(\varphi)
\]
for all $   \varphi \in \widehat A$.

\begin{lemma}
\label{lem-cpt-action1}
The $C^*$-algebra $C^*_r (L)$ acts by compact operators on the Hilbert module $\H_\sigma$.
\end{lemma}

\begin{proof}
The reduced $C^*$-algebra of $L=MA$ has the form
\[
C^*_r (L) \cong C^*_r (M) \otimes C^*_r (A) \cong C^*_r (M) \otimes C_0 (\widehat A) ,
\]
where the first isomorphism uses the product structure of $L$ and the second uses the  Fourier transform for the abelian group $A$.  The first factor in $C^*_r (M) \otimes C_0 (\widehat A)$ acts on $C_0(\widehat A, H_\sigma)$ through compact operators on $H_\sigma$, and the second factor acts through pointwise multiplications.  So the reduced $C^*$-algebra of $L$ acts through the $C^*$-algebra $C_0(\widehat A, \Compact (H_\sigma))$.  This is the $C^*$-algebra of compact operators on $\H_\sigma = C_0(\widehat A, H_\sigma)$.
\end{proof}

\begin{remark}
In fact the action map is a surjective homomorphism from $C^*_r (L)$ onto $\Compact (\H_\sigma)$. \end{remark}

\begin{corollary}
\label{cor-compact-act}
The $C^*$-algebra $C^*_r (G)$ acts as compact operators on the Hilbert module $\Ind _P^G \H_\sigma$. 
\qed
\end{corollary}

\section{Decomposition of the Reduced C*-Algebra}
\label{sec-decomp}

In this section we shall use  the analysis of the tempered dual, carried out mostly by Harish-Chandra and Langlands, to decompose the reduced group $C^*$-algebra into a direct sum of component $C^*$-algebras.   The decomposition is   well known, but not especially well documented.  In any case, the arguments are quite simple and fit well into the Hilbert module context.
 
\begin{definition}
\label{def-component-alg}
Let $P$ be a parabolic subgroup of $G$, and let $\sigma$ be an irreducible, square-integrable representation of the compactly generated part of the Levi factor of $P$.
Denote by 
\[
C^*_r (G) _{P,\sigma}  \subseteq \Compact (\Ind _P^G \H_\sigma)
\]
the image of the $C^*$-algebra $C^*_r (G)$ under its action as compact operators  on the Hilbert module $\Ind_P^G \H_\sigma$.  See Corollary~\ref{cor-compact-act}.
\end{definition}

Our aim is to show that the natural quotient homomorphisms from $C^*_r (G)$ to the component algebras $C^*_r(G)_{P,\sigma}$ determine an isomorphism of $C^*$-algebras
\[
\Csr(G)\stackrel \cong \longrightarrow \bigoplus _{[P,\sigma]} C^*_r(G)_{P,\sigma} ,
\]
where the sum is over representatives of suitable equivalence classes of para\-bolic subgroups $P$ and irreducible square-integrable representations $\sigma$.

First, we shall describe the equivalence relation used above on pairs $(P,\sigma)$.  

\begin{definition} 
\label{def-associate-class}
Two pairs $(P_1,\sigma_1)$ and $(P_2,\sigma_2)$, each consisting of a (standard) parabolic subgroup and an irreducible, square-integrable representation of the compactly generated part   of the Levi factor of the parabolic, are \emph{associate} if there is an element of $G$ that conjugates the Levi factor of $P_1$ to the Levi factor of $P_2$, and conjugates $\sigma_1$ to a representation unitarily equivalent to $\sigma_2$.
\end{definition}

The relevance of this concept of equivalence stems from the following result, which is the first of several substantial theorems in representation theory that we shall merely quote.

\begin{theorem}\textup{(See \cite[Section 11]{HC71}.)}
\label{thm-associate-reps}
If $(P_1,\sigma_1)$ and $(P_2,\sigma_2)$ are associate, then every parabolically induced representation $\Ind^G _{P_1} (\sigma_1 \otimes \varphi_1)$  is unitarily equivalent to some parabolically induced representation $\Ind^G _{P_2} (\sigma_2 \otimes \varphi_2)$.
\end{theorem}

\begin{remark}
In fact it is also true that the    representations of $G$ on the Hilbert  modules $\Ind _{P_1}^G \H_{\sigma_2}$ and $\Ind^G_{P_2}\H_{\sigma_2}$ are unitarily equivalent, as we shall note in the next section, but this is a more difficult result.  Theorem~\ref{thm-associate-reps} is proved by computing that the characters of  $\Ind_{P_1}^G    (\sigma_1\otimes \varphi_1)$ and $\Ind_{P_2}^G (\sigma_2\otimes\varphi_2)$ are equal, but this technique does not apply in the Hilbert module case. \end{remark}

Next, we shall need the following theorem of Harish-Chandra \cite[Section 36]{MR0219666} (see 
  also \cite[Section 7.7]{MR929683} for an exposition) that is the counterpart, for tempered representations of real reductive groups, of Bernstein's uniform admissibility theorem   \cite[Section 1.4]{BernsteinPoly}.  Hence the title we shall give it here.

\begin{theorem}\textup{(Uniform admissibility.)}
\label{thm-unif-adm}
Let $G$ be a reductive group with maximal compact subgroup $K$.    Let $\tau$ be an irreducible representation of $K$.  
There are at most finitely many equivalence classes of irreducible, square-integrable representations of the compactly generated part of $G$ whose restrictions to $K$ contain $\tau$ as a subrepresentation.
\end{theorem}

\begin{corollary} 
\label{cor-unif-adm}
 Let $L=MA$ be a Levi subgroup of $G$ and form the direct sum Hilbert space representation $\oplus _{[\sigma]} H_\sigma$ of $M$, indexed by a set of representatives of the unitary equivalence classes of the irreducible, square-integrable representations of $M$. The action of $C^*_r(M)$ on   $\oplus_{[\sigma]}  H_\sigma$ is through compact operators.
 \end{corollary}

\begin{proof}
As $p$ ranges over the isotypical projections associated to irreducible representations of $K\cap M$ (a maximal compact subgroup of $M$)  the subspaces $C^*_r(M) p $ span a dense subspace of $C^*_r (M)$. But according to Theorem~\ref{thm-unif-adm} the elements of $C^*_r (M)p$ act  as the zero operator in all but finitely many of the Hilbert spaces $H_\sigma$.  The result follows from this and from Theorem \ref{thm-cpt-action1}.
\end{proof}

 \begin{corollary}
Form the direct sum Hilbert $C_0(\widehat{A})$-module
 \[
   \bigoplus _{[P,\sigma]}  \Ind_P^G \H_\sigma 
   \]
   indexed by a set of representatives of the associate classes of pairs $(P,\sigma)$.
   The action of $C^*_r (G)$ on this Hilbert module is through compact operators.
   \end{corollary}
   
\begin{proof}
First,  fix a standard Levi subgroup, and consider only those parabolics $P$ with that Levi factor. 
As in the proof of  Lemma \ref{lem-cpt-action1}, it follows from Corollary~\ref{cor-unif-adm} that 
the action of $C^*_r (L)$ on the orthogonal direct sum Hilbert module 
\[
\bigoplus_{[\sigma]}  \H_\sigma = \bigoplus_{[\sigma]} C_0(\widehat A , H_\sigma)
\]
 is through compact operators.  It follows from 
  Proposition~\ref{prop-compact-action} and Lemma~\ref{lem-cpt-tensor-one} that $C^*_r (G)$ acts on 
  $\oplus_{[P,\sigma]}  \Ind_P^G \H_\sigma $ through compact operators.  
 The full result follows from the fact that there are only finitely many standard Levi subgroups.
\end{proof}

We obtain from the corollary a homomorphism of $C^*$-algebras
    \[
\Csr(G) \longrightarrow  \bigoplus _{[P,\sigma]} \Compact   (\Ind_P^G \H_\sigma   ),
\]
and so by definition a homomorphism
\begin{equation}
\label{eq-dir-sum-components1}
\Csr(G) \longrightarrow  \bigoplus _{[P,\sigma]} C^*_r (G)_{P,\sigma} .
\end{equation}

\begin{remark}\label{rem-dir-sum}
Of course the important point about (\ref{eq-dir-sum-components1}) is that the image lies within the $C^*$-algebraic direct sum, consisting of families of elements $a_\sigma \in C^*_r (G)_{P,\sigma}$ with $\lim_{\sigma\to \infty} \| a_\sigma \| =0$. See for instance \cite[p.6]{Lance}.
\end{remark}

Our next task is to compute the image of the homomorphism (\ref{eq-dir-sum-components1}).   For this we shall need the following important theorem of Langlands on the disjointness of principal series representations. References for this result are \cite[p. 142ff.]{MR1011897} and \cite{HC71}. See also \cite[Theorem 14.90]{Knapp1}.

\begin{theorem}\label{thm-Disjointness}\textup{(Disjointness)}
  If two principal series representations 
  \[
  \Ind _{P_1}^ G (\sigma_1 \otimes \varphi_1) \quad \text{and} \quad \Ind _{P_2} ^G (\sigma_2 \otimes \varphi_2)
  \]
   share an irreducible constituent, then there is an element of $G$ that conjugates the Levi factor of $P_1$ to the Levi factor of $P_2$, and conjugates $\sigma_1\otimes \varphi_1$ to a representation of $P_2$ that is unitarily equivalent to $\sigma_2 \otimes \varphi_2$.   
\end{theorem}

We shall also need to apply some elementary facts from $C^*$-algebra representation theory to (\ref{eq-dir-sum-components1}).

\begin{lemma}
\label{lem-easy-cstar1}
The irreducible representations of the $C^*$-algebra $C^*(G)_{P,\sigma}$, viewed as irreducible representations of $G$ through the quotient mapping 
\[
C^*_r (G) \longrightarrow C^*_r (G)_{P,\sigma},
\]
 are precisely the irreducible constituents of the principal series representations $\Ind_P^G  ( \sigma\otimes \varphi) $, as $\varphi$ ranges over $\widehat A$.
\end{lemma}

\begin{proof}
If $C$ is any $C^*$-subalgebra of a  $C^*$-algebra $B$, then every irreducible representation of $C$ is equivalent to the restriction of an irreducible representation of $B$ to a $C$-invariant, $C$-irreducible subspace (see \cite[Proposition 2.10.2]{DixmierEnglish}). In the present case, where $C=C^*_r (G)_{P,\sigma}$ and $B=\Compact (\Ind_P^G \H_\sigma)$, the irreducible representations of $B$ are the natural representations given by evaluation at $\varphi\in \widehat A$ on the Hilbert spaces $\Ind_P^G H_{\sigma\otimes\varphi}$. Indeed, $\Ind_P^G \H_\sigma$ gives a Morita equivalence between $B$ and the commutative $\Cs$-algebra $C_0(\widehat{A})$, so the irreducible representations of  $B$ are exactly those induced from irreducible representations of $C_0(\widehat{A})$,  which are in turn given by evaluation maps. To conclude, observe that $\Ind_P^G \H_\sigma\otimes_{C_0(\widehat{A})}\C_\varphi\simeq\Ind_P^G H_{\sigma\otimes\varphi}$ as Hilbert spaces, so that once restricted to $C$, this is precisely the principal series representation $\Ind_P^G  (\sigma\otimes \varphi)$.
\end{proof}

\begin{definition}
Let  $\varphi\colon A \to B$ be a surjective homomorphism of $C^*$-algebras. The \emph{support} of $\varphi$ is the subset of the dual $\widehat A$ consisting of all irreducible representations that factor through $\varphi$ (that is, they vanish on the kernel of $\varphi$).
\end{definition}

\begin{lemma}
\label{lem-two-ideals}
If $\varphi_1 \colon A \to B_1$ and $\varphi_2 \colon A \to B_2$ are surjective homomorphisms of $C^*$-algebras, and if the support of $\varphi_1$ is disjoint from the support of $\varphi_2$, then the homomorphism
\[
\varphi_1 \oplus \varphi_2 \colon A \longrightarrow B_1 \oplus B_2 
\]
is surjective.  
\end{lemma}

\begin{proof}
To say that the supports are disjoint is to say that no irreducible representation can vanish on the kernels of both $\varphi_1$ and $\varphi_2$.  But this can only happen when the algebraic sum of the kernels, which is in any case a closed, two-sided ideal in $A$,  is equal to $A$ (otherwise the quotient $C^*$-algebra $A/(J_1+J_2)$ would have a non-zero irreducible representation).    But elementary linear algebra shows that if $J_1 +J_2 = A$, then the 
natural projection map  
\[
\varphi_1 \oplus \varphi_2 \colon A \longrightarrow A/J_1 \oplus A/J_2 
\]
is surjective.  
\end{proof}

\begin{corollary}
\label{cor-n-ideals}
If $\varphi_k\colon A \to B_k$ for $k=1,\dots, n$  are surjective homomorphisms of $C^*$-algebras, and if the supports of the $\varphi_k$ are pairwise disjoint, then the homomorphism
\[
\bigoplus _k \varphi_k   \colon A \longrightarrow \bigoplus_k B_k 
\]
is surjective. \qed
\end{corollary}

\begin{lemma}
\label{lem-many-ideals}
Let  $\{ \varphi_\alpha \colon A \to B_\alpha\} $ be a   family of surjective homomorphisms of $C^*$-algebras. Assume that
\begin{enumerate}[\rm (a)]
\item the supports of the homomorphisms $\varphi_\alpha$ are pairwise disjoint, and 
\item the direct sum $\oplus _\alpha \varphi_\alpha$  maps $A$ into  $\oplus_\alpha B_\alpha$.  That is, 
\[
\lim _{\alpha \to \infty} \| \varphi_\alpha(a) \| = 0
\]
for every $a\in A$.
\end{enumerate}
Then the  homomorphism
\[
\bigoplus _\alpha \varphi_\alpha  \colon A \longrightarrow \bigoplus _\alpha B_\alpha  
\]
 that results from \textup{(}b\textup{)}  is  surjective. \end{lemma}

\begin{proof}
Fix any index $\alpha_0$ and organize the direct sum $C^*$-algebra as 
\[
\bigoplus_\alpha B_\alpha  = B_{\alpha_0} \oplus \bigoplus_{\alpha\ne \alpha_0} B_\alpha.
\]
To prove the lemma it suffices to show that the image of the homomorphism $\oplus  \varphi_\alpha$  contains all elements in the direct sum that are zero in the second term.
Let $J $ be the kernel of $\varphi_{\alpha_0}$.  The image of $J $ in $\oplus _{\alpha \ne \alpha_0} B_\alpha$ is a closed $C^*$-subalgebra, as indeed is the image of any $C^*$-algebra homomorphism.  But it follows from Corollary~\ref{cor-n-ideals} and the assumption (b) in the lemma that the image is also a two-sided  ideal.   This image ideal must be all of $\oplus _{\alpha \ne \alpha_0} B_\alpha$, for if it wasn't, the quotient $C^*$-algebra would have an irreducible representation $\pi$.  Its equivalence class, viewed as a point in $\widehat A$,  would lie in the support of $\varphi_{\alpha_0}$, and also in the support of $\varphi_\alpha$ for some $\alpha\ne \alpha_0$, contradicting disjointness of supports.
\end{proof}

Combining the Langlands Disjointness Theorem with these elementary observations we arrive at the following result: 

\begin{proposition}
Form the $C^*$-algebra direct sum 
\[
\bigoplus _{[P,\sigma]} C^*_r (G)_{P,\sigma}
\]
indexed by a set of representatives of the associate classes of pairs $(P,\sigma)$.  The quotient mappings from $C^*_r (G)$ into each summand determine a  $C^*$-algebra homomorphism
\[
\Csr(G) \longrightarrow  \bigoplus _{[P,\sigma]} C^*_r (G)_{P,\sigma} ,
\]
and moreover this homomorphism is surjective.
\end{proposition}

\begin{proof}
It follows from Theorem \ref{thm-Disjointness} that condition (a) in Lemma \ref{lem-many-ideals} is satisfied. As for (b), it holds by Remark \ref{rem-dir-sum}.
\end{proof}

We need one final result from representation theory in this section: 
 
\begin{theorem}
\label{thm-hc-principle}
\textup{(See \cite[Lemma 4.10]{MR1011897} or \cite{MR0453929}.)}
 Every tempered irreducible representation of $G$ may be realized as a subrepresentation of a principal series  representation.
\end{theorem}

\begin{proposition}
\label{prop-reductive-cstar-structure1}
Let $G$ be a real reductive group.  The  $C^*$-algebra homomorphism
\[
\Csr(G)\longrightarrow \bigoplus _{[P,\sigma]}  C^*_r(G)_{P,\sigma}
\]
is an isomorphism.
\end{proposition}

\begin{proof}
If the kernel was non-zero, then there would be an irreducible representation of $C^*_r (G)$ which did not vanish on it, and which therefore was distinct from any parabolically induced representation, or any subrepresentation of a parabolically induced representation, contrary to Theorem~\ref{thm-hc-principle}.
\end{proof}

\begin{remark}
It obviously follows from all of the above that the $C^*$-algebra $C^*(G)_{P,\sigma}$ depends only on the associate class of the pair $(P,\sigma)$, up to isomorphism.  This will be made more explicit in the next section.
\end{remark}

\section{Structure of  the Component C*-Algebras}
\label{sec-group-alg-structure}
 
In this section we shall determine the structure of the individual component $C^*$-algebras $C^*_r (G)_{P,\sigma}$ in Definition~\ref{def-component-alg} (or at any rate, enough of the structure of these algebras for our purposes).  Once again, to do so we shall rely on some substantial results in representation theory, namely Theorems~\ref{thm-intertwiner-existence} and \ref{thm-hc-completeness} below.

 Let $P_1$ and $P_2$ be standard parabolic subgroups of $G$ with Levi factors $L_i=M_{P_i}A_{P_i}$, and suppose given irreducible square-integrable unitary representations $\sigma_i$ of $M_{P_i}$ for $i\in\left\lbrace1,2 \right\rbrace $.

\begin{theorem}
\label{thm-intertwiner-existence}
If $w\in K$ conjugates $L_1$ into $L_2$, and if $\operatorname{Ad}_w ^* \sigma_1 \simeq \sigma_2$, then there is an $\Ad_w^*$-twisted unitary isomorphism  of Hilbert modules 
\[
U_w \colon \Ind _{P_1}^G   \H_{\sigma_1} \longrightarrow  \Ind _{P_2}^G   \H_{\sigma_2}
\]
that covers the isomorphism 
\[
\operatorname{Ad}_w ^* \colon C_0(\widehat A_{P_1})\longrightarrow  C_0(\widehat A_{P_2})
\]
in the sense of Definition \ref{def-twisted-action} and that commutes with the action of $C^*_r (G)$. 
It is unique up to a composition with an inner automorphism.
\end{theorem}

\begin{proof}
Each Hilbert module $\Ind _{P_i}^G \H_{\sigma_i}$ can be seen as a continuous field of Hilbert spaces over $\widehat A_{P_i}$. The existence of the operators $U_w$ then follows from that of the so-called Knapp-Stein normalized intertwiners, constructed in \cite{KS2} and providing unitary equivalences between principal series representations $\Ind _{P_1}^G \sigma_1\otimes \varphi$ and $\Ind _{P_2}^G \sigma_2\otimes \operatorname{Ad}_w ^*\varphi$ (see for instance \cite[Proposition 8.5.(v)]{KS2}).

As for the uniqueness assertion, let $P$ be $P_1$ or $P_2$. There is a dense open set of characters  $\varphi \in \widehat A_P$ for which the unitary Hilbert space representation $\Ind_P^G H_{\sigma\otimes \varphi}$ is irreducible. This is due to Bruhat \cite{Bruhat} if $P$ is minimal and to Harish-Chandra \cite{HC3} in general. See also \cite{MR670190}. In any case it follows from the more advanced Theorem~\ref{thm-hc-completeness} below.   Given two $\Ad_w^*$-twisted unitary automorphisms, the composition of one with the inverse of the other is a unitary Hilbert module automorphism $\Ind _P^G \H_\sigma$ that intertwines the action of $C^*_r(G)$.   Again viewing the Hilbert module as a continuous field of Hilbert spaces over $\widehat A_P$, our unitary automorphism is then a continuous family of unitary self-intertwiners  of the representations $\Ind _P^G H_{\sigma\otimes \varphi}$.
\end{proof}

\begin{remark}
Another approach to Knapp-Stein theory in the context of Hilbert modules consists in constructing $w$-twisted unitary operators directly at the level of $\Csr(G/N)$. That point of view was adopted in \cite{PC*entrelacSL2} and \cite{MR3246931} where explicit unitary intertwiners were obtained in the case of the special linear group.
\end{remark}

Consider now a single  parabolic subgroup $P\subseteq G$ and associated Levi subgroup $L$, and form the  group
\begin{equation}
\label{eq-L-weyl-group}
W  =  N_K(L)   /  K\cap L  =  N_G(L)/L .
\end{equation}
It is a \emph{finite} group; see \cite[Chap. V]{Knapp1}.  Next, fix an irreducible square-integrable unitary representation $\sigma$ of the compactly generated part  $M$   of $L$.  The group $W$ acts as outer automorphisms of $M$, and hence it acts on the set of equivalence classes of representations of $M$.  We define $W_\sigma$ to be the isotropy group of the equivalence class of $\sigma$:
\[
W_\sigma = \{ w \in N_K(L) : \operatorname{Ad}_w ^* \sigma \simeq \sigma\, \} /  K\cap L .
\]
The group $W$, and hence in particular the subgroup $W_\sigma$ acts as a group of automorphisms of the vector group $A \subseteq L$.

According to Theorem~\ref{thm-intertwiner-existence} there is a  group extension
\[
1 \longrightarrow \operatorname{Inner}(\Ind _P^G  \H_\sigma)  \longrightarrow \widetilde W _\sigma \longrightarrow W _\sigma \longrightarrow 1.
\]
where
\begin{enumerate}[\rm (a)]
\item  $\operatorname{Inner}(\Ind _P^G  \H_\sigma)$ is the group of inner unitary automorphims of $\Ind _P^G  \H_\sigma$. It is isomorphic to the unitary group of the multiplier algebra of $C_0(\widehat A)$, or in other words the group of modulus-one continuous functions on $\widehat A$.
\item $\widetilde W_\sigma$ is the group of all twisted unitary  automorphisms of $\Ind _P^G  \H_\sigma$ associated with elements of the normalizer group $N_K(L)$ that fix $\sigma$ up to equivalence, as in Theorem~\ref{thm-intertwiner-existence}.
\end{enumerate}
In particular, the self-intertwiners associated to the elements of $W_\sigma$ by Theorem~\ref{thm-intertwiner-existence} yield a projective unitary  action in the sense of Definition \ref{def-proj-act} of $W_\sigma$ on $\Ind _P^G   \H_\sigma$, hence an action homomorphism of the form
\begin{equation}
\label{eq-sigma-series}
\Csr(G) \longrightarrow  \Compact   (\H_\sigma   )^{W_\sigma}.
\end{equation}
 We are going to show that this map is surjective.

 \begin{definition}
 \label{def-small-isotropy}
Given $\varphi\in \widehat A$, let 
\[
W_{\sigma, \varphi} = \{\,  w \in W_\sigma : \operatorname{Ad}^*_w (\varphi) = \varphi\, \} .
\]
\end{definition}

For each $w\in W_{\sigma,\varphi}$ the twisted unitary automorphism $U_w$ of $\Ind_P^G \H_\sigma$ appearing in Theorem \ref{thm-intertwiner-existence} restricts to an actual unitary automorphism $U_{w,\varphi}$ of the fiber $\Ind_P^G (\sigma\otimes \varphi)$. 

\begin{definition}
We denote  by  $I(\sigma, \varphi)$ the  finite-dimensional $C^*$-algebra of operators on the Hilbert space of the  principal series representation $\Ind_P^G ( \sigma\otimes \varphi)$ generated by the Knapp-Stein intertwiners $U_{w,\varphi}$ associated with the elements of the finite group $W_{\sigma, \varphi}$.
\end{definition}

We start by making note of the representation theory of the  fixed point algebra (compare the discussion  in \cite[5.4.13]{DixmierEnglish}).
 
\begin{lemma}
\label{lem-reps-of-kw}
Let  $\varphi\in \widehat A$  and  $p\in I(\sigma,\varphi)$ be a minimal projection. 
\begin{enumerate}[\rm (a)]
\item The $C^*$-algebra   $\Compact (\Ind _P^G  \H_\sigma)^{W_\sigma}$ is represented irreducibly on the range of $p$. 
\item  Every irreducible representation of  $\Compact (\Ind _P^G  \H_\sigma)^{W_\sigma}$ arises this way, up to unitary equivalence.  \item The representations associated to two minimal projections 
\[
p_1 \in I(\sigma, \varphi_1) \quad \text{and} \quad p_2 \in I(\sigma, \varphi_2)
\]
are equivalent if and only if there is some element $w\in W_\sigma$ such that $\operatorname{Ad}^*_w (\varphi_1) = \varphi_2$ and such that the projection $\operatorname{Ad}^*_w (p_1)\in I(\sigma,\varphi_2)$ is equivalent to $p_2$. \qed
\end{enumerate}
\end{lemma}

The following result is known as Harish-Chandra's Completeness Theorem.

\begin{theorem}
\label{thm-hc-completeness}
\textup{\cite[Theorem 38.1]{HC3}} Let $\sigma$ be an irreducible square-integrable unitary representation representation of $M$, and let $\varphi$ be a unitary character of $A$.
The finite-dimensional $C^*$-algebra $I(\sigma, \varphi)$ is the full commutant of the parabolically induced representation  $\Ind_P^G ( \sigma\otimes \varphi)$.
\end{theorem}

\begin{proof}
See \cite[Corollary 9.8]{KS2} and \cite[Theorem 14.31]{Knapp1}.
\end{proof}

 \begin{proposition}
 The action homomorphism
 \[
 \Csr(G) \longrightarrow  \Compact   (\Ind _P^G  \H_\sigma   )^{W_\sigma}
 \]
 is surjective, and therefore $C^*_r (G)_{P,\sigma} =\Compact   (\Ind _P^G  \H_\sigma   )^{W_\sigma}$.
 \end{proposition}

\begin{proof}
It is evident from Lemma~\ref{lem-reps-of-kw} that the $C^*$-algebra 
\[
B = \Compact (\Ind _P^G  \H_\sigma)^{W_\sigma}
\]
 is postliminal, meaning that  the  image of $B$ in any irreducible representation is the $C^*$-algebra of compact operators on the representation Hilbert space.  We may therefore invoke  Dixmier's version of the Stone-Weierstrass Theorem \cite[Theorem 11.1.8]{DixmierEnglish} to prove surjectivity.   According to Dixmier's theorem, it suffices to prove that the irreducible representations of  $B$ pull back to  irreducible representations of $C^*_r (G)$, and that inequivalent irreducible representations of $B$ pull back to inequivalent irreducible representations of $C^*_r (G)$.

It follows from Harish-Chandra's Completeness Theorem that every irreducible representation of $B$ does indeed pull back to an irreducible representation of $C^*_r(G)$, and that for a given unitary character $\varphi$ of $A$, the irreducible representations of $B$ associated to inequivalent minimal projections in $I(\sigma, \varphi)$ remain inequivalent when pulled back to $C^*_r(G)$.   Part (b) of Lemma~\ref{lem-reps-of-kw} ensures that irreducible representations of $B$ associated to unitary characters of $A$ in distinct $W_\sigma$ orbits remain  inequivalent when pulled back to $C^*_r (G)$.\end{proof}

The isomorphism of Proposition \ref{prop-reductive-cstar-structure1} can now be rephrased as follows:

  \begin{theorem}
\label{thm-reductive-cstar-structure2}
Let $G$ be a real reductive group.  The combined action homomorphism
\begin{equation}\label{eq-cstarG-structure}
\Csr(G)\longrightarrow \bigoplus _{[P,\sigma]} \Compact (\Ind _P ^G \H_{\sigma} )^{W_{\sigma}}
\end{equation}
is an isomorphism of $C^*$-algebras. \qed
\end{theorem}

\begin{remark} 
\label{rem-wass}
A refined version of this  decomposition   appears in  Wassermann's short note \cite{NoteWassermann}. It incorporates additional information, due to Knapp and Stein, about the structure of the stabilizers $W_\sigma$ (see \cite{MR0422512}, \cite{KS2}, and \cite[XIV \S 9]{Knapp1}), who showed that the group $W_{\sigma,\varphi}$ admits a semidirect product decomposition
\[
W_{\sigma,\varphi}=W'_{\sigma,\varphi}\rtimes R_{\sigma,\varphi},
\]
 in which the factor $R_{\sigma,\varphi}$, called the \emph{$R$-group} attached to $(\sigma,\varphi)$, consists of those elements that actually contribute nontrivially to the intertwining algebra of $\Ind_P^G ( \sigma\otimes \varphi)$. The full group  $W_\sigma$ can also be written as semidirect product 
 \begin{equation*}
 \label{eq-Wsemidir}
 W_\sigma=W'_\sigma\rtimes R_\sigma ,
 \end{equation*}
and Wassermann notes that 
\[
\Compact (\Ind _P ^G \H_{\sigma} )^{W_{\sigma}}\underset{\text{Morita}}{\sim}C_0(\widehat{A}_P/W'_\sigma)\rtimes R_\sigma.
\]
This is needed  in   the computation of the $K$-theory of $\Csr(G)$, which was Wassermann's main concern. See \cite[\S 4]{MR1292018} for an account of the $K$-theoretic aspects of $C^*_r (G)$ (the Connes-Kasparov and Baum-Connes conjectures).
The theory of the $R$-group will not be required for our purposes. The decomposition of Theorem \ref{thm-reductive-cstar-structure2}  will be  enough.
\end{remark}

\begin{example}
The structure of the reduced $\Cs$-algebra of real-rank one groups was elucidated in  \cite{MR0476913}. Let us consider the basic case of $G=\mathrm{SL}(2,\R)$ (see \cite[\S4]{MR0476913} and \cite[p.256]{MR1292018}).

Up to association there are two parabolic subgroups---the group $G$ itself and the group   $P$ of upper triangular matrices.  The group $M$ in the decomposition $P=MAN$ consists of $\pm I$, while $A$ is isomorphic to $\R_+$.
The Weyl group consists of only two elements, both of which fix each of the two representations $\sigma_0$ and $\sigma_1$ of $M$ (the first being the trivial representation).

According to Theorem \ref{thm-reductive-cstar-structure2} the reduced $\Cs$-algebra decomposes as
\[
\Csr(G)\cong  \bigoplus _{[\pi]} \Compact (H_\pi)  \oplus C_0\bigl (\R, \Compact(H_{\sigma_0})\bigr )^{\mathbb Z_2} 
\oplus  C_0\bigl (\R, \Compact(H_{\sigma_1})\bigr )^{\mathbb Z_2} ,
\]
where the first direct sum   is indexed by the discrete series of $G$.  
\end{example}

\begin{remark}
A more refined analysis, as in Remark~\ref{rem-wass}, shows that 
\[
 C_0\bigl (\R, \Compact(H_{\sigma_0})\bigr )^{\mathbb Z_2}
 \underset{\text{Morita}}{\sim}
 C_0(\R/\Z_2)
 \]
 and
 \[
 C_0\bigl (\R, \Compact(H_{\sigma_1})\bigr )^{\mathbb Z_2}
 \underset{\text{Morita}}{\sim}C_0(\R)\rtimes\Z_2.\]
\end{remark}

\section{Structure of the Universal Principal Series}
\label{sec-bimodule-structure}

In this section we shall work with a fixed standard parabolic subgroup 
\begin{equation}
\label{eq-parabolic}
P=LN  
\end{equation}
of a real reductive group $G$.  We shall determine the structure of the correspondence $C^*_r (G/N)$ using the information about reduced group  $C^*$-algebras that we presented in   Sections~\ref{sec-decomp} and \ref{sec-group-alg-structure}.  

Applying the results in those sections to the real reductive group $L$ rather than $G$, we find that the action of $L$ on its principal series representations determines an isomorphism
\begin{equation}
\label{eq-structure-of-L}
C^*_r (L)\stackrel \cong\longrightarrow   \bigoplus _{[ Q,\sigma]}  
 \Compact  \bigl (\Ind _Q ^L \H_\sigma    \bigr )^{W_{\sigma}} .
\end{equation}
To review,  the sum is over associate classes of pairs $(Q,\sigma)$ consisting of a standard parabolic subgroup $Q\subseteq L$ and an irreducible, square-integrable representation of the compactly generated part of the Levi factor of $Q$.    

As we indicated in (\ref{eq-dir-sum-decomp-E1}), the above direct sum decomposition of $C^*_r(L)$  gives rise to a direct sum decomposition of any Hilbert $C^*_r(L)$-module, and our first task is to understand these summands in the case of $C^*_r (G/N)$.
According to (\ref{eq-dir-sum-decomp-E2}),  if $\E$ is any Hilbert $C^*_r (L)$-module, then the summands have the form
\begin{equation}
\label{eq-E-summands}
\Compact \bigl ( \Ind_Q^L \H_\sigma, \E\otimes _{C^*_r(L)} \Ind_Q^L \H_\sigma\bigr)^{W_\sigma } ,
\end{equation}
so we shall need to compute the tensor product $C^*_r (G/N)\otimes _{C^*_r (L)} \Ind_Q^L  \H_\sigma$.

To this end, we shall first make explicit what we mean by \emph{standard} in the context of the group $L$.  If $A$ is the given maximal abelian positive-definite subgroup of $G$, then $A\cap L$ is a maximal abelian positive-definite subgroup of $L$.  Similarly if $\mathfrak a_{+}$ is  a chamber defining a family of standard parabolic subgroups for $G$, then $\mathfrak a_{+}\cap \mathfrak l$ is contained in a chamber defining a family of standard parabolic subgroups for $L$.  We shall make compatible choices in this way.
Having made these choices, it follows from \cite[Lemma 4.1.17]{MR632407} that:

\begin{lemma}
\label{lem-para-para}
If $P=LN_P$, as in   \textup{(\ref{eq-parabolic})}, and if   
 $Q$ is any standard parabolic subgroup of $L$, then the product $QN_P$ is a standard parabolic subgroup of $G$.
Moreover, if we denote by $N_Q$ the unipotent radical of the standard parabolic subgroup $Q\subseteq L$, then the unipotent radical of $QN_P$ is $N_QN_P$. 
\end{lemma}

Define a convolution product 
\begin{equation}
\label{eq-ind-in-stages}
C_c ^\infty (G/N_P) \otimes C_c^\infty (L/N_Q) \longrightarrow C_c^\infty (G/N_PN_Q)
\end{equation}
 by
\[
\begin{aligned}
( f_1  * f _2) (x)  &  = \int _L \delta (\ell) ^{-\frac 12} f_1(x\ell^{-1}N_P) f_2 (\ell N_Q ) \, d\ell  \\
	& =  \int _L \delta (\ell) ^{\frac 12} f_1(x\ell N_P) f_2 (\ell^{-1} N_Q) \, d\ell,\\
	\end{aligned}
\]
where $\delta=\delta^G_P$ is the character \eqref{eq-delta-def} involved in the definition of $\Csr(G/N_P)$.
This product factors through the tensor product over $C_c^\infty (L)$ and has dense range. Moreover, denoting the Levi component of $Q$ by $J$ so that $Q=JN_Q$ and $QN_P=JN_QN_P$, it is compatible with the $C_c^\infty(J)$-valued inner products, as follows:

\begin{lemma} \textup{(Induction in Stages.)}
The convolution product \textup{(\ref{eq-ind-in-stages})} determines a unitary isomorphism of Hilbert $C^*_r (J)$-modules 
\[
C^*_r (G/N_P)\otimes _{C^*_r (L)} C^*_r (L/N_Q) \stackrel \cong \longrightarrow C^*_r (G/N_PN_Q).
\]
that is compatible with the left actions of $C^*_r (G)$ on both sides. \qed
\end{lemma}

\begin{remark}
Induction in stages for representations of \Cs-algebras was described in the original work of Rieffel \cite[Theorems 5.9 and 5.11]{RieffelIRC*A}.
The lemma above reflects relations amongst successively induced principal series representations that can be found in \cite[Proposition 4.1.18]{MR632407}.
\end{remark}

From this it is easy to compute the Hilbert module tensor product that appears in (\ref{eq-E-summands}).

\begin{corollary}
\label{cor-ind-stages}
Let $Q$ be a standard parabolic subgroup of $L$ with Levi factor $J$, and let $\sigma$  be an irreducible, square-integrable representation of the compactly generated part of $J$.
There is  a unitary isomorphism of Hilbert $C^*_r (J)$-modules 
\[
C^*_r (G/N)\otimes _{C^*_r (L)} \Ind_Q^L  \H_\sigma \stackrel \cong \longrightarrow \Ind_{QN}^G \H_\sigma.
\]
compatible with the left actions of $C^*_r (G)$.
\end{corollary}

\begin{proof} This is an immediate consequence of the above lemma and the fact that 
$  \Ind_Q^L  \H_\sigma  = C^*_r (L/N_Q)\otimes _{C^*_r (J)} \H_\sigma$.
\end{proof}

Applying the results of Section~\ref{sec-hilb-mod} we obtain the following description of the $C^*$-algebraic universal principal series.

\begin{theorem}
\label{thm-decomp-module}
There is a unitary  isomorphism of Hilbert modules
\[
C^*_r (G/N)  \stackrel\cong   \longrightarrow 
\bigoplus _{[ Q, \sigma]} 
 \Compact  \bigl ( \Ind_Q^L\H_\sigma, \Ind_{QN}^G \H_\sigma    \bigr )^{W_{\sigma}}
\]
with the following properties:
\begin{enumerate}[\rm (a)]
\item The left action of $C^*_r (G)$ on $C^*_r (G/N)$ corresponds under the isomorphism to the left action of $C^*_r (G)$ on each principal series Hilbert module $\Ind _{QN}^G \H_\sigma$.
\item The right action of $C^*_r (L)$ on  $C^*_r (G/N)$ corresponds under the isomorphism to the left action of $C^*_r (L)$ on each principal series Hilbert module $\Ind _Q^L \H_\sigma$.
\item The summands on the right hand side are orthogonal to one another, and the Hilbert module inner product on any one of them is 
\[
\langle S, T\rangle = S^*T,
\]
where the operator $S^*T \in \Compact (\Ind _Q^L \H_\sigma )^{W_\sigma}$ is to be viewed as an element of $C^*_r (L)$ via the isomorphism \textup{(\ref{eq-structure-of-L})}. \qed
\end{enumerate}
\end{theorem}
  
\begin{example}\label{ex-decomp-Pmin}
In the case of a \emph{minimal} parabolic subgroup $P = LN= MAN$, one has 
\[
\Csr(L)\cong \bigoplus_{[\sigma]\in \widehat  M}   C_0\bigl (\widehat A, \Compact (H_\sigma)\bigr ) .
\] 
Note that there are no proper parabolic subgroups of $L$, and $M$ is compact, so that each $H_\sigma$ is in fact finite-dimensional.   The decomposition of  Theorem~\ref{thm-decomp-module} reduces to
\begin{equation*}
\Csr(G/N) \cong \bigoplus _{\sigma\in\widehat M }\Compact (\H_\sigma ,\Ind_P^G \H_\sigma) \cong 
 \bigoplus _{\sigma\in\widehat M }  C_0\bigl ( \widehat A, \Compact (H_\sigma ,\Ind_P^G H_\sigma) \bigr )
 .
\end{equation*}
Note that the groups $W_\sigma(L)$ are trivial in this case.
\end{example}

\begin{remark}
Results analogous to Theorem \ref{thm-decomp-module} at the Hilbert space level are presented in the final chapter of \cite{MR1170566} in connection with Whittaker functions. See in particular 15.9.3 for the Plancherel decomposition of the quasiregular representation of $G$ on $L^2(G/N)$.

\end{remark}

\section{The Adjoint Hilbert Module}
\label{sec-adjoint-module}

\begin{definition}
\label{dej-adjoint-mod}
Let $A$ and $B$ be $C^*$-algebras, and let $\E$ be a Hilbert correspondence from $A$ to $B$.  We shall denote by $\E^*$ the complex conjugate vector space to $\E$, equipped with the following (algebraic) $B$-$A$-bimodule structure:
\[
b\cdot \overline e \cdot a = \overline{a^*e b^*}.
\]
This is the \emph{adjoint bimodule} to $\E$.
\end{definition}

The adjoint bimodule $\E^*$  will not typically carry the structure of a correspondence from $B$ to $A$, but in this section we shall show that when $\E= C^*_r (G/N)$, the {adjoint bimodule} $C^*_r (G/N)^*$ can be equipped with a $C^*_r(G)$-valued inner product that makes it a correspondence from $C^*_r(L)$ to $C^*_r (G)$. In the next section  we shall characterize this secondary inner product using the Plancherel theorem.

The significance of this fact is that interior tensor product with $C^*_r (G/N)^*$, as in Definition~\ref{def-interior-tp}, gives a functor of \emph{parabolic restriction} from tempered representations of $G$ to tempered representations of $L$.  In this paper we shall merely introduce the parabolic restriction functor.  In a subsequent paper \cite{CCH_adjoint} we shall show that the new functor is simultaneously left and right adjoint to the functor of parabolic induction between categories of unitary tempered Hilbert space representations. 

 The construction of the $C^*_r (G)$-valued inner product on $C^*_r (G/N)^*$ is in fact very straightforward, given the structure theory developed in the  last several sections.  
 
 Recall that according to Theorem~\ref{thm-reductive-cstar-structure2} there is an isomorphism
\begin{equation}
\label{eq-cstar-recap}
\Csr(G)\stackrel \cong \longrightarrow \bigoplus _{[P,\sigma]} \Compact (\Ind _P ^G \H_{\sigma} )^{W_{\sigma}} ,
\end{equation}
while according to Theorem~\ref{thm-decomp-module} there is an isomorphism
\begin{equation}
\label{eq-module-recap}
C^*_r (G/N)  \stackrel \cong   \longrightarrow 
\bigoplus _{[ Q, \sigma]} 
 \Compact  \bigl ( \Ind_Q^L\H_\sigma, \Ind_{QN}^G \H_\sigma    \bigr )^{W_{\sigma}} .
\end{equation}
The first thing to say about the $C^*_r(G)$-valued inner product is that, by definition, the summands in (\ref{eq-module-recap}), or rather the adjoint modules associated to them, will be orthogonal to one another.

As for the individual summands in (\ref{eq-module-recap}), the inner product of a pair of elements from the $[Q,\sigma]$-summand will lie in the $[QN,\sigma]$-summand of (\ref{eq-cstar-recap}); see Lemma~\ref{lem-para-para}.  In order to define this inner product  it will be helpful to refine a little the notation for groups of intertwining operators $W_\sigma$ that we have used up to now.  Let $J$ be a standard  Levi subgroup of $L$, where $L$ is in turn a standard Levi subgroup of $G$, and let $\sigma$ be an irreducible, square-integrable representation of the compactly generated part of $J$.  We shall now use the notations
 \[
 W_\sigma (G) = \{ w \in N_K(J) : \operatorname{Ad}_w ^* \sigma \simeq \sigma\, \} /  K\cap J  
\]
and 
 \[
 W_\sigma (L) = \{ w \in N_{K\cap L}(J) : \operatorname{Ad}_w ^* \sigma \simeq \sigma\, \} /  K\cap J ,
\]
which take into account the fact that $J$ may regarded as a standard Levi subgroup of either $G$ or $L$.  Of course,
 \[
 W_\sigma (L) \subseteq W_\sigma (G).
 \]
Suppose then we are given two elements in one summand of $C^*_r (G/N)^*$, say 
\[
\overline S_1, \overline S_2 \in \overline { \Compact    ( \Ind_Q^L\H_\sigma, \Ind_{QN}^G \H_\sigma      )^{W_{\sigma}(L)}  }
\]
(the complex conjugates are present as a result  of our definition of the adjoint module). We define 
\begin{equation}
\label{eq-G-average-inner-product}
\bigl \langle \overline S_1 , \overline S_2  \bigr \rangle _{C^*_r (G)} = 
\operatorname{Av}_{W_\sigma(G)} (S_1S_2^*) 
\in \Compact ( \Ind_{QN}^G\H_\sigma)^{W_\sigma(G)},
\end{equation}
where on the right-hand side we have taken the average over the action of the finite group $W_\sigma(G)$ on the $C^*$-algebra $ \Compact (\Ind _{QN} ^G \H_{\sigma} )  $.

The formula (\ref{eq-G-average-inner-product})  satisfies the algebraic requirements for a Hilbert module inner product. In addition, the $[Q,\sigma]$-summand that we are studying is complete in the norm associated to the inner product.  This is because 
\[
\bigl \langle \overline S  , \overline S  \bigr \rangle _{C^*_r (G)} = \operatorname{Av}_{W_\sigma(G)} (S S^*) \ge \frac {|W_\sigma(L)|}{|W_\sigma(G)|} SS^*,
\]
and so, since $\|SS^*\| = \| S^*S\|$, 
\begin{equation}
\label{eq-equivalent-norms}
\| S \|_{C^*_r (G/N)^*}^2 \ge  \frac {|W_\sigma(L)|}{|W_\sigma(G)|} \|S\|_{C^*_r (G/N)}^2 .
\end{equation}
Therefore we obtain a Hilbert $C^*_r(G)$-module structure on the $[Q,\sigma]$-summand of $C^*_r (G/N)^*$, as required.

To complete the construction we need to show that the (complex conjugate of the) Hilbert module direct sum (\ref{eq-module-recap}), which is the completion of the algebraic direct sum in the $C^*_r (G/N)$-norm, is also the Hilbert module direct sum in the $C^*_r (G/N)^*$-norm.

Each summand of the Hilbert module is supported, as either a Hilbert $C^*_r(G)$-module or a Hilbert $C^*_r(L)$-module, in a single summand of the reduced group $C^*$-algebra (that is, all the inner products lie in a single summand in the direct sum decomposition of the $C^*$-algebra). Moreover there is a  uniform bound on the number of Hilbert module summands that are supported in any given $C^*$-algebra summand.  Finally, in addition to   the inequality (\ref{eq-equivalent-norms}) we also have the inequality
\[
 \|S\|_{C^*_r (G/N)}^2 \ge \| S \|_{C^*_r (G/N)^*}^2 
\]
in each Hilbert module summand.  So the  $C^*_r (G/N)$-norm in each summand  is bounded uniformly by a multiple of the $C^*_r (G/N)^*$-norm, and vice versa.   It follows that the $C^*_r (G/N)$-norm and the $C^*_r (G/N)^*$-norm are bounded by multiples of each other on the algebraic direct sum, so the completion in the two norms agree, as required.

\begin{definition}
Let $G$ be a real reductive group and let $P=LN$ be a parabolic subgroup of $G$. 
We shall  denote the correspondence from $C^*_r (L)$ to $C^*_r(G)$ just constructed by  
\[
C^*_r (N\backslash G) = C_r^* (G/N) ^*  .
\]
\end{definition}

\begin{remark}
The  notation reflects the fact that the $L$-$G$-bimodule $C_r^*(G/N)^*$ can be viewed as a completion of $C_c^\infty (N\backslash G)$, with the left $L$-action and right $G$-action defined analogously to our approach  in Section~\ref{sec-parabolic-induction} to the actions on $C_c^\infty (G/N)$. One  associates to a function $f\in C_c^\infty (N\backslash G)$ the function 
\[
f^*\colon gN\mapsto \overline{f(Ng^{-1} )}
\]
 on $G/N$.    
 \end{remark}
 
 \begin{definition}
 \label{def-parabolic-restriction}
If $H$ is a Hilbert space carrying a tempered unitary representation of $G$, or in other words a  nondegenerate representation of the $C^*$-algebra $C^*_r (G)$, then the \emph{parabolic restriction} of $H$ is the Hilbert space
\[
\Res^G _P H = C^*_r (N\backslash G) \otimes _{C^*_r (G)} H ,
\]
together with the   tempered unitary representation of $L$ that it carries.
\end{definition}
 
 Here are the results of some easy sample calculations, along with some remarks about parabolic restriction as it is understood from a more standard representation-theoretic point of view.  
 
 In each case the calculation is carried out using the explicit forms of $C^*_r(G)$ and $C^*_r (G/N)$ that we have determined in the paper.  In other words the calculations, though short, are far from basic, in that they presuppose a great deal of representation theory of the sort relied upon in this paper. 
 
 \begin{example}
 \label{ex-discr-series}
The first thing to be said  is that if an irreducible representation $\pi$ is square-integrable on the compactly generated part of $G$, then $\Res^G_P H_\pi =0$ for all proper parabolic subgroups $P\subseteq G$.  This is because $C^*_r (G)$ can be written as a sum of two complementary ideals, one acting trivially on $H_\pi$ and the other acting trivially on $C^*_r (N\backslash G)$.   

In contrast, the usual (left) adjoint to parabolic induction considered in representation theory (the space of $\mathfrak n$-coinvariants of a $(\mathfrak g, K)$-module) is nonzero in this and indeed any admissible case, by Casselman's famous subrepresentation theorem.
\end{example}

 \begin{example}Consider next a principal series representation $\Ind_P^G H_{\sigma\otimes \varphi}$.  If the isotropy group $W_{\sigma,\varphi}$ of Definition~\ref{def-small-isotropy}  is trivial (and so, for example, by Theorem~\ref{thm-hc-completeness}  the representation is irreducible), then 
\[
\Res^G_P \Ind ^G_P H_{\sigma\otimes \varphi} \cong \bigoplus _{w\in W} H_{w(\sigma\otimes \varphi)}
\]
This is consistent with the standard situation in representation theory.
\end{example}

\begin{example}
More interesting, perhaps is the case of a principal series representation for which the intertwining group $W_{\sigma,\varphi}$ is large.  Consider for example  the base of the spherical principal series.  Here one has 
\[
\Res^G_P \Ind ^G_P \C_0 \cong \C_0 .
\]
As in Example~\ref{ex-discr-series}, this  is smaller than the space   of $\mathfrak n$-coinvariants of the associated $(\mathfrak{g}, K)$-module.
 \end{example}

As we have already mentioned, the central fact about the functor of parabolic restriction, which we shall establish in \cite{CCH_adjoint}, is as follows. 

\begin{theorem}
\label{thm-double-adjoint}
 The functor $\Res_P^G$,  from tempered unitary representations of $G$ to tempered unitary representations of $L$,  is both  left  and right adjoint to the functor $\Ind^G_P$ of parabolic induction.
\end{theorem}

The proof is not   difficult, granted the structure theory for $C^*_r (G/N)$ that we have developed here, but it involves a separate set of operator-algebraic ideas, and it is for this reason that we have chosen to defer it.  
 
\begin{remark}
\label{rem-bernstein}
Let us close this section by   making a few informal comments for the benefit of those who are familiar with Bernstein's second adjoint theorem for smooth representations of reductive $p$-adic groups \cite{BernsteinPoly}.  Bernstein begins by observing that there is a natural candidate unit map for his adjunction.  It takes the form of a bimodule map
\[
C_c^\infty(L) \longrightarrow C_c(\overline N \setminus G ) \otimes _{C_c^\infty (G)} C_c^\infty (G/N) ,
\]
where $\overline N$ is the unipotent subgroup opposite to $N$, and it is associated to the inclusion of $\overline N \cdot L \cdot N$ as an open subset in $G$.  

Passing to   $C^*$-algebras and Hilbert modules (in either the real or the $p$-adic contexts) one can ask, does Bernstein's map extend to completions?  It does not.  However it is reasonably well-behaved as an unbounded operator (it is regular \cite[Chapter 9]{Lance}).  

Another manageable problem is the appearance of $\overline N$ in place of a second copy of $N$ (as we use in this paper and its sequel).   The $C^*$-correspondences associated to the two different unipotents are isomorphic. 

After switching $\overline N$ for $N$, the unit map that arises from Theorem~\ref{thm-double-adjoint} can be viewed as a bounded transform (in roughly the sense of \cite[Chapter 10]{Lance}). But we emphasize that it  exists only at the Hilbert space level, not the Hilbert module level. 

In conclusion, then, Bernstein's unit, and hence his adjunction,  is distinct  from that of Theorem~\ref{thm-double-adjoint}, but related to it. 
 
 A perhaps more interesting and more fundamental observation is that at the level of Harish-Chandra's Schwartz space (see the next section)  and its associated bimodules,   Bernstein's unit is indeed well-defined.   Moreover we have checked for $G=\mathrm{SL}(2,\R)$  that the counterpart of Bernstein's theorem is \emph{true} (and we believe it is true generally). See \cite{CH_cb} for some preliminary computations in this direction; a fuller treatment of these matters is in preparation.  
\end{remark}

\section{Relation to the Plancherel Formula}

We shall close this paper with a calculation that relates the inner product on the adjoint correspondence $C^*_r (N\backslash G)$ to the Plancherel formula.

It will be convenient to work with Harish-Chandra's Schwartz space $\mathcal C (G)$. This is a Fr\'echet space of smooth functions on $G$ that contains as a dense subspace the smooth and compactly supported functions. The Schwartz space is included continuously as a dense subspace in  both $L^2 (G)$ and $C^*_r (G)$.   See \cite[Part I]{MR0219666} and, for an exposition,  \cite[Chapter 7]{MR929683}.  In every irreducible representation of $C^*_r (G)$ the functions in $\mathcal C (G)$ act as Hilbert-Schmidt operators.

We begin by reviewing the Plancherel formula for $G$ (and for this purpose we shall not need to fix yet a parabolic subgroup $P\subseteq G$).  

Choose a representative  $(P,\sigma)$ for each of the associate classes $[P,\sigma]$ for $G$, as in Definition~\ref{def-associate-class}.   There is a Langlands decomposition $P = M_PA_PN_P$, and, as we noted earlier the group  $A_P$, which consists entirely of positive-definite matrices,  is isomorphic to its Lie algebra  via the exponential map. So $A_P$ carries   the structure of a vector space, and  we can speak its   space of Schwartz functions in the ordinary sense of harmonic analysis.  The same goes for the unitary (Pontrjagin) dual $\widehat A_P$.  By a  \emph{tempered} measure on $\widehat A _P$ we   mean a smooth measure  for which integration extends to a continuous linear functional on the Schwartz space.

Finally, let  us write 
\[
\pi_{\sigma,\varphi} = \Ind ^G_P (\sigma \otimes \varphi)
\]
for the representation of $G$ parabolically induced from the representation $\sigma\otimes \varphi $ of $L=M_PA_P$.   Using this terminology and notation, the general structure of Harish-Chandra's Plancherel formula for the group $G$ \cite{HC3} is as follows (see \cite[Chapter 13]{MR1170566} for an exposition). 

\begin{theorem}
\label{thm-plancherel1}
There are unique smooth, tempered, $W_\sigma$-invariant measures $m_{P,\sigma}$ on the spaces $ \widehat A_P$ such that 
\begin{equation*}
\|f\|^2 _{L^2 (G)}  = \sum _{[P,\sigma]}   
	\int _{\widehat A _P} \|\pi_{\sigma,\varphi} (f) \|^2 _{\text{H-S}}\, dm_{P,\sigma} (\varphi),
\end{equation*}
for every $f\in \mathcal C (G)$.  
\end{theorem}

As $\varphi\in \widehat A_P$ varies, the Hilbert spaces $\Ind_P^G H_{\sigma\otimes \varphi}$ can be identified with one another as representations of $K$.  Denote by $\Ind _P^G H_\sigma$  this common Hilbert space (as we have done earlier)  and form the Hilbert space tensor product
\begin{equation}
\label{eq-hilb-tensor-prod}
 L^2 \bigl (\widehat A_P , m_{P,\sigma}\bigr ) \otimes \mathcal L^2 \bigl (\Ind _P^G H_\sigma \bigr ) ,
\end{equation}
where $ \mathcal L^2   (\Ind _P^G H_\sigma   ) $  denotes the Hilbert space of Hilbert-Schmidt operators on    $\Ind_P^G H_\sigma$.

If $f\in \mathcal C (G)$ is \emph{$K$-finite}, meaning that its left and right $K$-translates span a finite-dimensional subspace of $\mathcal C (G)$, then   for every pair $(P,\sigma)$ the function 
\[
\widehat A_P \ni \varphi \mapsto \pi_{\sigma, \varphi}(f) \in \mathcal L^2 (\Ind _P^G H_\sigma)
\]
 is a Schwartz function from $\widehat A_P$ into a \emph{$K$-finite part} of $\mathcal L^2 (\Ind _P^G H_\sigma)$, meaning a finite-dimensional subspace that is invariant under the left and right actions of $K$.  We can regard the function as an element of the Hilbert space tensor product (\ref{eq-hilb-tensor-prod}). 

\begin{definition}
\label{def-fourier}
Let $f\in \mathcal C(G)$ be a $K$-finite function.  Its  \emph{Fourier transform} is the element 
of the direct sum Hilbert space 
\[
\bigoplus _{[P,\sigma]} L^2 \bigl (\widehat A_P , m_{P,\sigma}\bigr ) \otimes \mathcal L^2 \bigl (\Ind _P^G H_\sigma \bigr ) 
\]
determined by the operator-valued functions $\varphi \mapsto \pi_{\sigma,\varphi}(f)$.
\end{definition}

The  Plancherel formula can be reformulated in these terms, as follows. 

\begin{theorem}
\label{thm-plancherel3}
The Fourier transform, defined initially on $K$-bi-finite functions in $\mathcal C(G)$, extends to an isometric linear map
\begin{equation*}
L^2 ( G) \longrightarrow \bigoplus _{[P,\sigma]} L^2 \bigl (\widehat A_P , m_{P,\sigma}\bigr ) \otimes \mathcal L^2 \bigl (\Ind _P^G H_\sigma \bigr ) .
\end{equation*}
\end{theorem}
 
 We shall also need the following determination of the range of the Fourier transform.  See \cite[Chap.~3, Section 1]{MR697608}. 
 
\begin{theorem}
\label{thm-plancherel-range}
The  Fourier transform is a Hilbert space isometry from $L^2 (G)$ onto the   Hilbert subspace 
\[
\bigoplus _{[P,\sigma]}  \left [ L^2 \bigl (\widehat A_P , m_{P,\sigma}\bigr ) \otimes \mathcal L^2 \bigl (\Ind _P^G H_\sigma \bigr )\right ] ^{W_\sigma (G)}
\subseteq 
\bigoplus _{[P,\sigma]} L^2 \bigl (\widehat A_P , m_{P,\sigma}\bigr ) \otimes \mathcal L^2 \bigl (\Ind _P^G H_\sigma \bigr )  .
\]
\end{theorem}

We shall use this fact to  calculate the adjoint to the Fourier transform in Definition~\ref{def-fourier}.
 
\begin{definition}
\label{def-wave-packet}
Let  $h$ be a Schwartz-class function from $\widehat A_P$ into a $K$-finite part of $\mathcal L^2 (\Ind_P^G H_\sigma)$.  The \emph{wave packet} associated to $h$ is the scalar function    
\begin{equation}
\label{eq-wave-packet}
\check h (g) =  \int _{\widehat A_P} \Trace \bigl (\pi_{\sigma,\varphi} (g^{-1}) h(\varphi)\bigr )  \, dm_{P,\sigma} (\varphi).
\end{equation}
on the group $G$.
\end{definition}

A fundamental theorem of Harish-Chandra asserts that wave packets are Schwartz functions on $G$.     See for example \cite[Theorems~12.7.1 and 13.4.1]{MR1170566}.

\begin{theorem}
\label{thm-wave-packet}
The wave packets \textup{(\ref{eq-wave-packet})} associated to the Schwartz-class functions from $\widehat A_P$ into the $K$-finite parts of $\mathcal L^2 (\Ind_P^G H_\sigma)$ all belong  to  the Harish-Chandra Schwartz space $\mathcal C(G)$.
\end{theorem}

We can now carry out  the following   crucial computation involving wave packets, which is a simple consequence of Theorem~\ref{thm-plancherel-range}. 
\begin{proposition}
If $f$ is a Harish-Chandra Schwartz function on $G$, and if $h$ is a Schwartz function from $A_P$ to the $K$-finite part of $\mathcal L^2 ( \Ind_P^G H_{\sigma})$, then 
\[
\langle  f, \check h\rangle _{L^2 (G)} =  \langle \hat f , h \rangle _{ L^2  (\widehat A  , m_{P,\sigma}  ) \otimes \mathcal L^2  (\Ind_P^G H_{\sigma} )} .
\]
\end{proposition}

\begin{proof}
We calculate that 
\[
\begin{aligned}
\langle  f, \check h\rangle _{L^2 (G)} 
	&= \int_{G} \overline{f(g)} \, \check h (g) \, dg \\
	&= \int_{G} \overline{f(g)} \, \int_{\widehat A_P}\Trace\bigl  (\pi_{\sigma,\varphi} (g^{-1}) \cdot h(\varphi )\bigr ) \, dm_{P,\sigma}(\varphi )   dg \\
	&=  \int_{\widehat A_P}\Trace\Bigl  (\int_{G} \overline{f(g)} \pi_{\sigma,\varphi}(g^{-1})\, dg \cdot  h(\varphi )\Bigr ) \, dm_{P,\sigma}(\varphi ) \\	
	&=  \int_{\widehat A_P}\Trace\bigl  ( \pi_{\sigma,\varphi} (f)^* \,  h(\varphi )\bigr ) \, dm_{P,\sigma}(\varphi ) \\
	&= \langle \hat f, h\rangle _{  L^2  (\widehat A_P  , m_{P,\sigma}  ) \otimes \mathcal L^2  (\Ind_P^G H_{\sigma}  )} ,
	\end{aligned}
\]
as required.
\end{proof}

\begin{corollary}
The wave-packet operator $h\mapsto \check h$, defined on Schwartz functions with values in $K$-finite parts of the Hilbert-Schmidt spaces $\mathcal L^2 (\Ind _P^G H_\sigma )$,  extends to a bounded linear operator 
\[
\bigoplus _{[P,\sigma]} 
	L^2 \bigl (\widehat A_P , m_{P,\sigma}\bigr ) \otimes \mathcal L^2 \bigl (\Ind _P^G H_\sigma \bigr ) 
\longrightarrow 
	L^2 (G)
\]
that is adjoint to the Fourier transform. \qed
\end{corollary}

Now if   $h\in  L^2 \bigl (\widehat A_P , m_{P,\sigma}\bigr ) \otimes \mathcal L^2 \bigl (\Ind _P^G H_\sigma \bigr ) 
$, then define $\operatorname{Av}_{W_\sigma}(h)$ to be the average over the action of the finite group $W_\sigma$:
\[
 \operatorname {Av}_{W_\sigma(G)} (h ) (\varphi)  = \frac{1}{ | W_\sigma | } \sum _{w\in W_\sigma} w \bigl ( h(w^{-1}(\varphi))\bigr ) .
 \]
 The averaging operator 
 \[
  \operatorname {Av}_{W_\sigma (G)} \colon L^2 \bigl (\widehat A_P , m_{P,\sigma}\bigr ) \otimes \mathcal L^2 \bigl (\Ind _P^G H_\sigma \bigr ) 
  \longrightarrow 
  L^2 \bigl (\widehat A_P , m_{P,\sigma}\bigr ) \otimes \mathcal L^2 \bigl (\Ind _P^G H_\sigma \bigr ) 
  \]
  is the orthogonal projection onto the $W_\sigma(G)$-invariant part of the tensor product.
  
\begin{corollary}   
\label{cor-wave-packet-fmla}
The $[P,\sigma]$-component of the Fourier transform of the wave packet $\check h$ is given by the formula
\[
\pi_\varphi (\check h) =  \operatorname {Av}_{W_\sigma(G)} (h )(\varphi).
\]
The other components of the Fourier transform are zero. \qed
\end{corollary}
 
We are ready to give a formula for our   $C^*_r (G)$-valued inner product on the summand 
\[
\left [  \Compact  \bigl ( \Ind_Q^L\H_\sigma, \Ind_{QN}^G \H_\sigma    \bigr )^{W_{\sigma}(L)}\right ] ^*  \subseteq  C^*_r (G/N)^*
 \]
 as in (\ref{eq-module-recap}).   Write 
 \[
 \Compact  \bigl ( \Ind_Q^L\H_\sigma, \Ind_{QN}^G \H_\sigma    \bigr )^{W_{\sigma}(L)} \cong C_0 \bigl (\widehat A, \Compact(  \Ind_Q^L H_\sigma, \Ind_{QN}^G H_\sigma)\bigr )^{W_{\sigma}(L)} ,
 \]
 and suppose we are given two functions 
 \[
 S_1,S_2\in C_0 \bigl (\widehat A_P, \Compact(  \Ind_Q^L H_\sigma, \Ind_{QN}^G H_\sigma)\bigr )^{W_\sigma(L)}
 \]
 that are in fact  Schwartz functions from $\widehat A_P$ into a finite part of the Hilbert-Schmidt space $\mathcal L^2 (  \Ind_Q^L H_\sigma, \Ind_{QN}^G H_\sigma)$.
  
\begin{theorem}
Given $S_1$ and $S_2$ as above, the inner product 
\[
\langle \overline{S}_1, \overline{S}_2\rangle_{C^*_r (G)} \in C^*_r (G)
\]
 is  the following Harish-Chandra Schwartz function on $G$: 
\begin{equation}
\label{eq-wave-packets}
\langle \overline S_1,\overline S_2 \rangle_{C^*_r (G)} = \Bigl [  g \mapsto \int _{\widehat A_P} \Trace \bigl( S_2(\varphi)^* \pi_{\sigma,\varphi}(g^{-1}) S_1(\varphi) \bigr )\, d m_{P,\sigma} (\varphi) \Bigr ]
\end{equation}
\end{theorem}

\begin{proof}
By definition, the inner product is the unique element of $C^*_r (G)$ that is  equal to 
\[
 \operatorname {Av}_{W_\sigma(G)} (S_1 S_2^*) 
\]
in the $[QN,\sigma]$-component of the direct sum decomposition (\ref{eq-cstarG-structure}) of $C^*_r (G)$, and that is zero in the other components of the direct sum decomposition.  See (\ref{eq-G-average-inner-product}).  If we write the right-hand side of (\ref{eq-wave-packets}) as 
\[
 g \mapsto \int _{\widehat A_P} \Trace \bigl(  \pi_{\sigma,\varphi}(g^{-1}) S_1(\varphi) S_2(\varphi)^* \bigr )\, d m_{P,\sigma} (\varphi) 
 \]
 using the trace property, then we see from Corollary~\ref{cor-wave-packet-fmla} that this function has precisely the required property.
\end{proof}



\bibliographystyle{amsalpha}
\bibliography{biblio} 
\end{document}